\documentclass[reqno,11pt]{amsart}
\usepackage{amsmath,amsfonts,amssymb,amsxtra,latexsym,amscd,enumerate,amsthm,verbatim}

\usepackage{hyperref}
\usepackage{graphicx}

\usepackage[margin=1.40in]{geometry}
\numberwithin{equation}{section}

\newcommand{\R}{\mathbb{R}}

\numberwithin{equation}{section} 

\newtheorem{theorem}{Theorem}[section]
\newtheorem{lemma}[theorem]{Lemma}

\newtheorem{remark}[theorem]{Remark}

\begin{document}

\title{linear inviscid damping near monotone shear flows}


\author{Hao Jia}
\address{University of Minnesota}
\email{jia@umn.edu}


\begin{abstract}
{\small}
We give an elementary proof of sharp decay rates and the linear inviscid damping near monotone shear flow in a periodic channel, first obtained in \cite{dongyi}. We shall also obtain the precise asymptotics of the solutions, measured in the space $L^{\infty}$.

\end{abstract}

\maketitle

\setcounter{tocdepth}{1}


\section{Introduction}

Hydrodynamic stability is one of the oldest problems studied in partial differential equations and has been investigated by prominent figures such as Kelvin, Rayleigh, Orr among many others. The early works mostly focused on the issue of spectral stability of physically relevant flows, such as shear flows and circular flows, see e.g., \cite{Kelvin,Kirchhoff,Orr,Rayleigh}. 

In the case of monotone shear flows, which is our main interest in this paper, Faddeev \cite{Faddeev} studied the general spectral property of monotone shear flows, Lin \cite{Lin} obtained a sharp condition for the presence of unstable eigenvalues, and Stepin \cite{Stepin} proved a quantitative decay estimate of the stream function associated with the continuous part of the spectrum of the monotone shear flows, see also \cite{ZhiWu} for the optimal decay in the Couette case. 

Recently, inspired by the remarkable work of Bedrossian and Masmoudi \cite{BeMa} on the nonlinear asymptotic stability of shear flows close to the Couette flow in $\mathbb{T}\times\R$ (see also an extension \cite{IOJI} to $\mathbb{T}\times[0,1]$), optimal decay estimates for the linear problem received much attention, see e.g.  Zillinger \cite{Zillinger1,Zillinger2} and references therein for shear flows close to Couette.  In an important work, Wei, Zhang and Zhao in \cite{dongyi} obtained the optimal decay estimates for the linearized problem around monotone shear flows, under very general conditions. We also refer the reader to important developments for the linear inviscid damping in the case of non-monotone shear flows \cite{Dongyi2, Dongyi3} and circular flows \cite{Bed2,Zillinger3}. See also Grenier et al \cite{Grenier} for an approach using methods from the study of Schr\"odinger operators.

Our main goal in this paper is to provide an elementary alternative proof of the optimal decay rate for the linearized problem around monotone shear flows, first obtained in \cite{dongyi}. The main new idea is to use spaces that adapt precisely to the structure of singularities of the generalized eigenfunctions, and to treat the singular integrals using integration by parts arguments. We will discuss our approach in more details below.

The second goal is to identify the main terms in the aymptotic, measured in the stronger space $L^{\infty}$. The more precise understanding of the aymptotic may be useful for the nonlinear analysis, since the residue term, which is expected to decay in an integrable fashion, can in principle be controlled by cruder methods. See the remark below Lemma \ref{L10} for more discussions. 

The third goal is to clarify the role of the boundary effects in deciding the dynamics of the solutions. The fact that the boundary effect is significant, and can be an obstruction for scattering of the vorticity in high regularity spaces, has already been observed by Zillinger in \cite{Zillinger2} for shear flows close to Couette flow (i.e., linear shear). In our paper, the boundary effect can be clearly seen, since it contributes, in a relatively explicit way, to the main term in the asymptotics, see \eqref{L11} below. In addition, by tracking precisely the main terms, it seems clear that the boundary effects are the {\it only} obstruction to scattering in high Sobolev, and even Gevrey spaces. We remark that the Gevrey space control, in a suitably adapted coordinate system, is essential for proving nonlinear asymptotic stability, as demonstrated in \cite{Deng}. Hence, to extend the linear analysis to nonlinear analysis which is much more subtle and challenging, it appears necessary to assume conditions that allow the vorticity to be supported away from the boundary, as in \cite{IOJI} for the Couette flow. We plan to investigate this issue in another place.

\subsection{Equations and the main result}
We now turn to disuss in more details the main equations that we shall study. Consider the two dimensional linearized Euler equation around a shear flow $(b(y),0)$ in a periodic channel $ (x,y)\in \mathbb{T}\times[0,1]$:
\begin{equation}\label{Main1}
\begin{split}
&\partial_t\omega+b(y)\partial_x\omega-b''(y)u^y=0,\\
&{\rm div}\,u=0\qquad{\rm and}\qquad \omega=-\partial_yu^x+\partial_xu^y,
\end{split}
\end{equation}
with the natural non-penetration boundary condition $u^y|_{y=0,1}=0$. For the linearized flow, 
$$\int\limits_{\mathbb{T}\times[0,\,1]}u^x(x,y,t)\,dxdy \qquad {\rm and}\qquad \int\limits_{\mathbb{T}\times[0,\,1]}\omega(x,y,t)\,dxdy$$
are conserved quantities. In this paper, we will assume that
 $$\int_{\mathbb{T}\times[0,1]}u_0^x(x,y)\,dxdy=\int_{\mathbb{T}\times[0,1]}\omega_0\,dx dy=0.$$
 These assumptions can be dropped by adjusting $b(y)$ with a linear shear flow $C_0y+C_1$.
 Then one can see from the divergence free condition on $u$ that 
there exists a stream function $\psi(t,x,y)$ with $\psi(t,x,0)=\psi(t,x,1)\equiv 0$, such that 
\begin{equation}\label{eqS1}
u^x=-\partial_y\psi,\,\,u^y=\partial_x\psi.
\end{equation}
The stream function $\psi$ can be solved through
\begin{equation}\label{eq:equationStream}
\Delta\psi=\omega, \qquad \psi|_{y=0,1}=0.
\end{equation}

We summarize our equations as follows
\begin{equation}\label{main}
\left\{\begin{array}{ll}
\partial_t\omega+b(y)\partial_x\omega-b''(y)\partial_x\psi=0,&\\
\Delta \psi(t,x,y)=\omega(t,x,y),\qquad \psi(t,x,0)=\psi(t,x,1)=0,&\\
(u^x,u^y)=(-\partial_y\psi,\partial_x\psi),&
\end{array}\right.
\end{equation}
for $ t\ge0, (x,y)\in\mathbb{T}\times[0,1]$. Our goal is to understand the long time behavior of $\omega(t)$ as $t\to\infty$ with small regular initial $\omega_0$. 

The main conditions we shall assume on the shear flow $b(y)\in C^4([0,1])$ are:\\

 (1) For some $\vartheta\in(0,1/10)$, 
\begin{equation}\label{A}
\vartheta/100\leq|b'(y)|\leq 1/(100 \vartheta);
\end{equation}

(2) The linearized operator $\omega\to b(y)\partial_x\omega-b''(y)\psi$ has no embedded eigenvalues.\\

We shall discuss the assumption on the absence of embedded eigenvalues in more details in subsection \ref{idea} below.
\medskip

Our main result is the following theorem.
\begin{theorem}\label{thm}
Let $\omega$ be a smooth solution to \eqref{main} with associated velocity field $u=(u^x,u^y)$, stream function $\psi$, and initial data $\omega_0\in H^3(\mathbb{T}\times[0,1])$. Assume that $\int_{\mathbb{T}}\omega_0(x,y)\,dx=0$ and that $\omega_0$ belongs to the projection to the continuous spectrum of the operator $\omega\to b(y)\partial_x\omega-b''(y)\partial_x\psi$.
Set
\begin{equation}\label{Th1}
f(t,x,y):=\omega(t,x+b(y)t,y),\qquad \phi(t,x,y):=\psi(t,x+b(y)t,y).
\end{equation}
Then there exist functions $F(x,y), \Psi(x,y)\in L^{\infty}(\mathbb{T}\times[0,1])$ with
\begin{equation}\label{Th1.1}
\|F\|_{L^{\infty}(\mathbb{T}\times[0,1])}+\|\Psi\|_{L^{\infty}(\mathbb{T}\times[0,1])}\lesssim \big\|\omega_0\big\|_{H^3},
\end{equation}
such that the following statements hold:

(1) The normalized vorticity scatters:
\begin{equation}\label{Th2}
\lim_{t\to\infty}\|f(t,x,y)-F(x,y)\|_{L^{\infty}(\mathbb{T}\times[0,1])}=0;
\end{equation}
The normalized stream function satisfies the bounds
\begin{equation}\label{Th3}
\sup_{\alpha\in\{0,1\}}\|\partial_x^{\alpha}\phi(t,x,y)\|_{L^{\infty}(\mathbb{T}\times[0,1])}\lesssim t^{-2}\big\|\omega_0\big\|_{H^3},\qquad \|\partial_y\phi(t,x,y)\|_{L^{\infty}(\mathbb{T}\times[0,1]))}\lesssim t^{-1}\big\|\omega_0\big\|_{H^3}.
\end{equation}

(2) In addition, if $\omega_0$ vanishes on the boundary of the periodic channel, i.e., $\omega_0|_{y=0,1}=0$, then 
\begin{equation}\label{Th4}
\lim_{t\to\infty}\|\partial_{x,y}f(t,x,y)-\partial_{x,y}F(x,y)\|_{L^{\infty}(\mathbb{T}\times[0,1])}=0,
\end{equation}
and
\begin{equation}\label{Th5}
\lim_{t\to\infty}\Big[\|t^2\phi(t,x,y)-\Psi(x,y)\|_{L^{\infty}(\mathbb{T}\times[0,1])}+\|t\partial_y\phi(t,x,y)\|_{L^{\infty}(\mathbb{T}\times[0,1])}\Big]=0.
\end{equation}

 \end{theorem}

\begin{remark}
The assumption that $\int_{\mathbb{T}}\omega(t,x,y)\,dx=0$ is preserved by the flow and the general case can be reduced to this case by subtracting a shear flow.

More precise and technical version of the theorem can be obtained, with more explicit expressions for the function $\Psi$, see \eqref{Psi6}. We have not tried to track the optimal dependence on the regularity of the initial data, partly to keep the paper simpler to read, and also due to the fact that for the corresponding nonlinear analysis it is necessary to work in much smoother spaces (Gevrey type spaces) than Sobolev spaces and higher frequencies needs to be controlled using other arguments in any case. See \cite{BeMa,Deng,IOJI}.
\end{remark}

\begin{remark}
The identification of the main asymptotic term for $\phi$ in \eqref{Th5} may be useful for the analysis of nonlinear stability problems, as the residue term is expected to decay faster, in an integrable fashion. The more precise formula \eqref{L11.5} seems to suggest a modification of the main dynamics from the Couette case, which may be relevant for proving nonlinear stability for general monotone shear flows. 

\end{remark}

\begin{remark}
In general, it is necessary to assume that $\omega_0$ vanishes at $y=0,1$ for \eqref{Th4}-\eqref{Th5} to hold, see \eqref{L11}. The boundary effect (even when assuming $\omega_0$ to vanish on the boundary) prevents scattering in higher regularity spaces,  as already observed by Zillinger \cite{Zillinger2} for shear flows close to Couette. 


\end{remark}

\subsection{Main idea of proof}\label{idea}
We now briefly outline the strategy of the proof of \ref{thm}.

Taking Fourier transform in $x$ in the equation \eqref{main} for $\omega$, we obtain that
 \begin{equation}\label{F3.0}
 \partial_t\omega_k+ikb(y)\omega_k-ikb''(y)\psi_k=0,
 \end{equation}
 for $k\in\mathbb{Z}, t\ge0, y\in[0,1]$. In the above, $\omega_k$ and $\psi_k$ are the Fourier coefficients for $\omega,\psi$ respectively.
 
 For each $k\in\mathbb{Z}\backslash\{0\}$, we set for any $g\in L^2(0,1)$,
 \begin{equation}\label{F3.1}
 L_kg(y)=b(y)g(y)+b''(y)\int_0^1G_k(y,z)g(z)dz,
 \end{equation}
 where $G_k$ is the Green's function for the operator $k^2-\frac{d^2}{dy^2}$ on $(0,1)$ with zero Dirichlet boundary condition. Then \eqref{F3.0} can be reformulated as
 \begin{equation}\label{F3.2}
 \partial_t\omega_k+ikL_k\omega_k=0.
 \end{equation}
 The spectral property of $L_k$ is well understood, and the spectrum is in general consisted of the continuous spectrum $[b(0), b(1)]$ with possible embedded eigenvalues at the inflection points of $b(y)$, i.e. points $y_c$ where $b''(y_c)=0$, together with some discrete eigenvalues with nonzero imaginary part for small $k$ which can only accumulate at embedded eigenvalues, see for instance \cite{Faddeev}. The presence of embedded eigenvalues is a non-generic situation.

In this paper, we assume that there is no embedded eigenvalues, which is the generic situation. More precisely we assume that

\smallskip
{\sl ({\bf A}) The operator $f\in L^2\to b(y)f+b''(y)\int_0^1G_k(y,z)f(z)\,dz\in L^2$ has no eigenvalues with the value $b(y_0), y_0\in[0,1]$, for any $k\in\mathbb{Z}\backslash\{0\}$.}

\medskip 
 
Assume now that $\omega_0$ has trivial projection in the discrete modes. By standard theory of spectral projection, we then have
 \begin{equation}\label{F4}
 \begin{split}
 \omega_k(t,y)&=\frac{1}{2\pi i}\lim_{\epsilon\to0+}\int_{\R}e^{i\lambda t}\left[(\lambda+kL_k-i\epsilon)^{-1}-(\lambda+kL_k+i\epsilon)^{-1}\right]\omega_0\,d\lambda\\
 &=\frac{1}{2\pi i}\lim_{\epsilon\to0+}\int_{0}^1e^{-ikb(y_0) t}|b'(y_0)|\left[(-b(y_0)+L_k-i\epsilon)^{-1}-(-b(y_0)+L_k+i\epsilon)^{-1}\right]\omega_0\,dy_0.
 \end{split}
 \end{equation}
We then obtain
 \begin{equation}\label{F5}
 \begin{split}
 \psi_k(t,y)&=-\frac{1}{2\pi i}\lim_{\epsilon\to0+}\int_{0}^1e^{-ikb(y_0) t}|b'(y_0)|\int_0^1G_k(y,z)\\
 &\hspace{1in}\times\bigg\{\Big[(-b(y_0)+L_k-i\epsilon)^{-1}-(-b(y_0)+L_k+i\epsilon)^{-1}\Big]\omega_0\bigg\}(z)\,dz dy_0\\
 &=-\frac{1}{2\pi i}\lim_{\epsilon\to0+}\int_{0}^1e^{-ikb(y_0) t}|b'(y_0)|\left[\psi_{k,\epsilon}^{-}(y,y_0)-\psi_{k,\epsilon}^{+}(y,y_0)\right]dy_0.
 \end{split}
 \end{equation}
 In the above, 
 \begin{equation}\label{F6}
 \begin{split}
 &\psi_{k,\epsilon}^{+}(y,y_0):=\int_0^1G_k(y,z)\Big[(-b(y_0)+L_k+i\epsilon)^{-1}\omega_0\Big](z)\,dz,\\
 &\psi_{k,\epsilon}^{-}(y,y_0):=\int_0^1G_k(y,z)\Big[(-b(y_0)+L_k-i\epsilon)^{-1}\omega_0\Big](z)\,dz.
 \end{split}
 \end{equation}
 We note that $\psi_{k,\epsilon}^{+}(y,y_0), \psi_{k,\epsilon}^{-}(y,y_0)$ satisfy for $\iota\in\{+,-\}$
 \begin{equation}\label{F7}
 -k^2\psi_{k,\epsilon}^{\iota}(y,y_0)+\frac{d^2}{dy^2}\psi_{k,\epsilon}^{\iota}(y,y_0)-\frac{b''(y)}{b(y)-b(y_0)+i\iota\epsilon}\psi_{k,\epsilon}^{\iota}(y,y_0)=\frac{-\omega_0(y)}{b(y)-b(y_0)+i\iota\epsilon}.
 \end{equation}
 The idea is to analyze the optimal smoothness of the functions $\psi_{k,\epsilon}^{\iota}(y,y_0), \iota\in\{\pm\}$ in $y_0$, and use integration by parts in the formula \eqref{F5} to obtain decay in time for $\psi_k(t,y)$.

The main difficulty in analyzing the smoothness of the generalized eigenfunctions $\psi_{k,\epsilon}^{\iota}(y,y_0)$ for $\iota\in\{\pm\}$ is the presence of singularities, which is most obvious when $b''(y)=0$ (the Couette flow). In the case of Couette flow, the generalized eigenvalues can be explicitly solved in the form
\begin{equation}\label{F7.1}
\psi^{\iota}_{k,\epsilon}(y,y_0)\approx\int_0^1G_k(y,z)\frac{\omega_0^k(z)}{b(z)-b(y_0)+i\iota\epsilon}dz.
\end{equation}
\eqref{F7.1} can be analyzed directly and it follows that $\partial_{y_0}\psi_{k,\epsilon}^{\iota}(y,y_0)$ has a log singularity of the form $\log{(b(y)-b(y_0)+i\iota\epsilon)}$ and $\partial^2_{y_0}\psi_{k,\epsilon}^{\iota}(y,y_0)$ has a singularity of the form $1/(b(y)-b(y_0)+i\iota\epsilon)$. In particular, $\partial^2_{y_0}\psi_{k,\epsilon}^{\iota}(y,y_0)$ is no longer integrable as $\epsilon\to0$. Such singularities of the generalized eigenfunctions are the reason for the slow algebraic decay for the linearized flow. 

Our main new idea is to use norms which are adapted to the singularities of the generalized eigenfunctions. Such norms are carefully chosen and depend both on the spectral parameter $y_0$ and the smoothing parameter $\epsilon$. Classically, fixed spaces independent of $\epsilon$ are more often used. In our case, since the generalized eigenfunctions possess singularities that depend both on the spectral parameters and the smoothing parameters, the $y_0,\epsilon$ dependent spaces are more suitable in measuring the singularities.

The choice of the spaces which captures the precise nature of the singularities allow us to treat the singular factor $1/(b(z)-b(y_0)+i\epsilon)$, or powers of it, that appeared in the main eigenfunction equation \eqref{F7}, by simple integration by parts argument. 

The method is elementary and we are able to extract the main asymptotic term (not just the decay rate) in the strong $L^{\infty}$ space. Moreover, the effect of the boundary terms can be completely understood from our method, at least in principle. We expect similar methods to work in other settings, such as the linearized vortex problem, which might provide simpler proofs of the important and difficult results in \cite{Bed2}.

In this paper, we use the notation $A\lesssim B$ to denote $A\leq C B$ for a suitable constant $C>1$ which is independent of the parameters $k,y_0,\epsilon$. We also use $\langle y\rangle=\sqrt{1+y^2}$ for $y\in\R$.

\section{Boundedness of the operator $T_{k,y_0,\epsilon}$} 
For integers $k\in\mathbb{Z}\setminus\{0\}$, recall that the Green's function $G_k(y,z)$ solves  
\begin{equation}\label{eq:Helmoltz}
-\frac{d^2}{dy^2}G_k(y,z)+k^2G_k(y,z)=\delta_z(y),
\end{equation}
with Dirichlet boundary conditions $G_k(0,z)=G_k(1,z)=0$, $z\in [0,1]$. $G_k$ has the explicit formula 
\begin{equation}\label{eq:GreenFunction}
G_k(y,z)=\frac{1}{k\sinh k}
\begin{cases}
\sinh(k(1-z))\sinh (ky)\qquad&\text{ if }y\leq z,\\
\sinh (kz)\sinh(k(1-y))\qquad&\text{ if }y\geq z.
\end{cases}
\end{equation}

We note the following bounds for $G_k$
\begin{equation}\label{Gk1.1}
\begin{split}
\sup_{y\in[0,1], |A|\leq10}&\bigg[|k|^2\big\|G_k(y,z)(\log{|z-A|})^{m}\big\|_{L^1(z\in[0,1])}+|k|\big\|\partial_{y,z}G_k(y,z)(\log{|z-A|})^{m}\big\|_{L^1(z\in[0,1])}\bigg]\\
&\lesssim |\log{\langle k\rangle}|^m,\qquad {\rm for}\,\,m\in\{0,1,2,3,4,5\}.
\end{split}
\end{equation}

$G_k$ has the following symmetry
\begin{equation}\label{Gk2}
G_k(y,z)=G_k(z,y), \qquad {\rm for}\,\,k\in\mathbb{Z}\backslash\{0\}, y, z\in[0,1].
\end{equation}

Define
\begin{equation}\label{bX5}
G_k'(y,z)=\frac{1}{\sinh{k}}\left\{\begin{array}{lr}
                                         -k\cosh{(k(1-z))}\cosh{(ky)}, &0\leq y\leq z\leq 1;\\
                                         -k\cosh{(kz)}\cosh{(k(1-y))}, &1\ge y>z\ge0.
                                          \end{array}\right.
\end{equation}
By direct computation, we see $G_k'$ satisfies the bounds
\begin{equation}\label{Gk3.1}
\begin{split}
\sup_{y\in[0,1], |A|\leq10}&\bigg[\big\|G_k'(y,z)(\log{|z-A|})^{m}\big\|_{L^1(z\in[0,1])}+|k|^{-1}\big\|\partial_{y,z}G_k'(y,z)(\log{|z-A|})^{m}\big\|_{L^1(z\in[0,1])}\bigg]\\
&\lesssim |\log{\langle k\rangle}|^m,\qquad {\rm for}\,\,m\in\{0,1,2,3,4,5\}.
\end{split}
\end{equation}

We note that
\begin{equation}\label{Gk1.2}
\partial_y\partial_zG_k(y,z)=\partial_z\partial_yG_k(y,z)=\delta(y-z)+G'_k(y,z),\qquad{\rm for}\,\,y,z\in[0,1].
\end{equation}

Fix $\epsilon\in[-1/4,1/4]\backslash\{0\}, y_0\in[0,1],  k\in\mathbb{Z}\backslash\{0\}$. Define for each $f\in L^2(0,1)$ the operator
\begin{equation}\label{T1}
T_{k,y_0,\epsilon}f(y):=\int_0^1G_k(y,z)\frac{f(z)}{b(z)-b(y_0)+i\epsilon}dz.
\end{equation}

For each $\epsilon\in[-1/4,1/4]\backslash\{0\}, y_0\in[0,1], k\in\mathbb{Z}\backslash\{0\}$ and integers $1\leq m\leq 5$, define for any $f$
\begin{equation}\label{Y1'}
\begin{split}
\|f\|_{Y^{1,m}_{k,y_0,\epsilon}}:=&\big\|f\big\|_{L^{\infty}}+\left\|\frac{f'(y)/|k|}{(\log{(b(y)-b(y_0)+i\epsilon)}-\vartheta^{-1})^{1+m}}\right\|_{L^{\infty}}.
\end{split}
\end{equation}
For the sake of simplicity we use the convention that
\begin{equation}\label{Y1.1}
Y^1_{k,y_0,\epsilon}:=Y^{1,1}_{k,y_0,\epsilon}.
\end{equation}

We also need a stronger version of $Y^{1,m}_{k,y_0,\epsilon}$. Define for any $f$,
\begin{equation}\label{Z1'}
\|f\|_{Z^{1,m}_{k,y_0,\epsilon}}:=\big\|f\big\|_{L^{\infty}}+\,|k|^{-1}\inf_{f'=g(y)\log{(b(y)-b(y_0)+i\epsilon)}+h}\bigg[\big\|g\|_{Y^{1,m}_{k,y_0,\epsilon}}+\big\|h\big\|_{Y^{1,m+1}_{k,y_0,\epsilon}}\bigg].
\end{equation}
We shall use the convention that
\begin{equation}\label{Z1.1}
Z^1_{k,y_0,\epsilon}:=Z^{1,1}_{k,y_0,\epsilon}.
\end{equation}

Clearly for $y_0\in[0,1], k\in\mathbb{Z}\backslash\{0\}, \epsilon\in[-1/4,1/4]\backslash\{0\}$, with uniform constants,
\begin{equation}\label{inclu}
\|f\|_{Z^{1,m}_{k,y_0,\epsilon}}\lesssim \|f\|_{Y^{1,m}_{k,y_0,\epsilon}}.
\end{equation}

In addition, we shall need to work with singular functions. To capture the precise singular behavior of the generalized eigenfunctions, define for $\epsilon\in[-1/4,1/4]\backslash\{0\}, k\in\mathbb{Z}\backslash\{0\},y_0\in[0,1]$ and integers $1\leq m\leq 5$ the norm
\begin{equation}\label{X1.1}
\|f\|_{X^{1,m}_{k,y_0,\epsilon}}:=\inf\,\sum_{j=0}^m\|g_j\|_{Y^{1,m-j+1}_{k,y_0,\epsilon}},
\end{equation}
where the infimum is taken over all representations
\begin{equation}\label{X1}
\begin{split}
f(y):=\sum_{j=0}^mg_j(y)\big[\log(b(y)-b(y_0)+i\epsilon)\big]^j.
\end{split}
\end{equation}

We define a slightly stronger norm for $f=g(y)\big[\log{(b(y)-b(y_0)+i\epsilon)}-\vartheta^{-1}\big]$, 
\begin{equation}\label{X9'}
\left\|g(y)\big[\log{(b(y)-b(y_0)+i\epsilon)}-\vartheta^{-1}\big]\right\|_{X^2_{k,y_0,\epsilon}}:=\left\|g\right\|_{Z^1_{k,y_0,\epsilon}};
\end{equation}
define also for $f=g(y)/\big[b(y)-b(y_0)+i\epsilon\big]$,
\begin{equation}\label{X10}
\left\|\frac{g(y)}{b(y)-b(y_0)+i\epsilon}\right\|_{X^3_{k,y_0,\epsilon}}=|k|\|g\|_{ Z^{1}_{k,y_0,\epsilon}}.
\end{equation}

To simplify the notations, in this subsection we often suppress the dependence in $\epsilon\in[-1/4,1/4]\backslash\{0\}, k\in\mathbb{Z}\backslash\{0\},y_0\in[0,1]$ in the notations. For instance we use $T, X^1, Y^1, Z^1, X^2, X^3$ instead of $T_{k,y_0,\epsilon}, X^1_{k,y_0,\epsilon}, Y^1_{k,y_0,\epsilon}$, $Z^1_{k,y_0,\epsilon}$, $X^2_{k,y_0,\epsilon}$, $X^3_{k,y_0,\epsilon}$.

\begin{lemma}\label{bX1}
There exists a constant $C>1$ independent of $\epsilon\in[-1/4,1/4]\backslash\{0\}, k\in\mathbb{Z}\backslash\{0\},y_0\in[0,1]$ and integers $1\leq m\leq 5$, such that
\begin{equation}\label{bX1.0}
\left\|T_{k,y_0,\epsilon}\,f\right\|_{Z^{1,m}_{k,y_0,\epsilon}}\leq C|k|^{-1}\big[\log{\langle k\rangle}\big]^{m+2} \left\|f\right\|_{Y^{1,m}_{k,y_0,\epsilon}}.
\end{equation}
\end{lemma}

\begin{proof}
 We normalize so that $\|f\|_{Y^{1,m}}:=1$.

 We first make the simple observation that
\begin{equation}\label{bX2}
Tf(y)=-\int_0^1\partial_z\left[G_k(y,z)\frac{f(z)}{b'(z)}\right]\log{(b(z)-b(y_0)+i\epsilon)}\,dz.
\end{equation}
It follows from \eqref{bX2} and \eqref{Gk1.1} that
\begin{equation}\label{bX3}
\|Tf\|_{L^{\infty}}\lesssim|k|^{-1}\big[\log{\langle k\rangle}\big]^{m+2}.
\end{equation}
We calculate
\begin{equation}\label{bX4}
\begin{split}
\partial_y(Tf)(y)=&-\partial_y\int_0^1\partial_zG_k(y,z)\frac{f(z)}{b'(z)}\log{(b(z)-b(y_0)+i\epsilon)}\,dz\\
                          &- \partial_y\int_0^1G_k(y,z)\partial_z\left[\frac{f(z)}{b'(z)}\right]\log{(b(z)-b(y_0)+i\epsilon)}\,dz\\                    
                        =&-\frac{f(y)}{b'(y)}\log{(b(y)-b(y_0)+i\epsilon)}-\int_0^1G_k'(y,z)\frac{f(z)}{b'(z)}\log{(b(z)-b(y_0)+i\epsilon)}\,dz\\
                          &-\int_0^1\partial_yG_k(y,z)\partial_z\left[\frac{f(z)}{b'(z)}\right]\log{(b(z)-b(y_0)+i\epsilon)}\,dz.
\end{split}
\end{equation}

From the definitions \eqref{Z1'}, the lemma follows easily from \eqref{bX4} and \eqref{Gk1.1}.\\
\end{proof}

To obtain better regularity of $\psi^{+}_{k,\epsilon}(y_0,y), \psi^{-}_{k,\epsilon}(y_0,y)$, we shall also need estimates of the operator $T_{k,y_0,\epsilon}$ in the more singular norms $X^2_{k,y_0,\epsilon}$ and $X^3_{k,y_0,\epsilon}$.

\begin{lemma}\label{X11}
There exists a constant $C>1$, independent of $\epsilon\in[-1/4,1/4]\backslash\{0\}$, $k\in\mathbb{Z}\backslash\{0\},y_0\in[0,1]$, such that
\begin{equation}\label{X12}
\left\|T_{k,y_0,\epsilon}f\right\|_{ Y^1_{k,y_0,\epsilon}}\leq C|k|^{-1}\big[\log{\langle k\rangle}\big]^4\left\|f\right\|_{X^2_{k,y_0,\epsilon}}, \qquad {\rm for\,\,any\,\,}f\in X^2_{k,y_0,\epsilon}.
\end{equation}
Moreover,
\begin{equation}\label{X12.1}
\left\|\partial_yT_{k,y_0,\epsilon}f\right\|_{X^{1,2}_{k,y_0,\epsilon}}\leq C\big[\log{\langle k\rangle}\big]^4\left\|f\right\|_{X^2_{k,y_0,\epsilon}}, \qquad {\rm for\,\,any\,\,}f\in X^2_{k,y_0,\epsilon}.
\end{equation}
\end{lemma}

\begin{proof}
We normalize so that $\|f\|_{X^2}:=1$.
 
 
We first make the simple observation that for $f(y)=g(y)\big[\log{(b(y)-b(y_0)+i\epsilon)}-\vartheta^{-1}\big]$ with $\|g\|_{Z^1}\lesssim1$,  
\begin{equation}\label{X13}
\begin{split}
Tf(y)=&-\frac{1}{2}\int_0^1\partial_zG_k(y,z) \frac{g(z)}{b'(z)}\big[\log{(b(z)-b(y_0)+i\epsilon)}-\vartheta^{-1}\big]^2\,dz\\
         &-\frac{1}{2}\int_0^1G_k(y,z) \partial_z\left[\frac{g(z)}{b'(z)}\right]\big[\log{(b(z)-b(y_0)+i\epsilon)}-\vartheta^{-1}\big]^2\,dz.
         \end{split}
\end{equation}
Thus
\begin{equation}\label{X14}
\|Tf\|_{L^{\infty}}\lesssim|k|^{-1}\big[\log{\langle k\rangle}\big]^4.
\end{equation}

We now calculate
\begin{equation}\label{X15}
\begin{split}
\partial_y(Tf)(y)=&-\frac{g(y)}{2b'(y)}\big[\log{(b(y)-b(y_0)+i\epsilon)}-\vartheta^{-1}\big]^2\\
                          &-\int_0^1G'_k(y,z) \frac{g(z)}{2b'(z)}\big[\log{(b(z)-b(y_0)+i\epsilon)}-\vartheta^{-1}\big]^2\,dz\\
                          &-\int_0^1\partial_yG_k(y,z) \partial_z\left[\frac{g(z)}{2b'(z)}\right]\big[\log{(b(z)-b(y_0)+i\epsilon)}-\vartheta^{-1}\big]^2\,dz.
 \end{split}
\end{equation}

From \eqref{X14}-\eqref{X15}, the bounds \eqref{Gk1.1} and \eqref{Gk3.1}, it follows that  
$$\|Tf\|_{Y^1}\lesssim |k|^{-1}\big[\log{\langle k\rangle}\big]^4,$$
 which completes the proof of \eqref{X12}. \eqref{X12.1} follows from \eqref{X15} by definitions.\\
\end{proof}

We shall also need the following variant of Lemma \ref{X11}.
\begin{lemma}\label{bX1m}
There exists a constant $C>1$, independent of $\epsilon\in[-1/4,1/4]\backslash\{0\}$, $k\in\mathbb{Z}\backslash\{0\},y_0\in[0,1]$ and integers $1\leq m\leq 5$, such that
\begin{equation}\label{bX1.0m}
\left\|T_{k,y_0,\epsilon}\,f\right\|_{Y^{1,m}_{k,y_0,\epsilon}}\leq C|k|^{-1}\big[\log{\langle k\rangle}\big]^{m+3} \left\|f\right\|_{X^{1,m}_{k,y_0,\epsilon}}.
\end{equation}
Moreover,
\begin{equation}\label{bX1.1m}
\left\|\partial_yT_{k,y_0,\epsilon}\,f\right\|_{X^{1,m+1}_{k,y_0,\epsilon}}\leq C\big[\log{\langle k\rangle}\big]^{m+3} \left\|f\right\|_{X^{1,m}_{k,y_0,\epsilon}}.
\end{equation}
\end{lemma}
The proof follows similar calculations as in the proof of Lemma \ref{X11}; we omit the repetitive details.\\

We now prove estimates on $T_{k,y_0,\epsilon}$ in the most singular norms we shall use.
\begin{lemma}\label{bX17}
There exists a constant $C>1$, independent of $\epsilon\in[-1/4,1/4]\backslash\{0\}$, $k\in\mathbb{Z}\backslash\{0\},y_0\in[0,1]$, such that for any 
$$f=\frac{g}{b(y)-b(y_0)+i\epsilon}$$
it holds that
\begin{equation}\label{X17}
\begin{split}
\bigg\|T_{k,y_0,\epsilon}f-\frac{g(y)}{|b'(y)|^2}\log{(b(y)-b(y_0)+i\epsilon)}+\mathcal{B}\,\bigg\|_{Y^1_{k,y_0,\epsilon}}\leq C|k|^{-1}\big[\log{\langle k\rangle}\big]^4\left\|f\right\|_{X^3_{k,y_0,\epsilon}};
\end{split}
\end{equation}
and
\begin{equation}\label{X17.0}
\left\|\partial_{y}T_{k,y_0,\epsilon}f(y)-\frac{g(y)}{b'(y)(b(y)-b(y_0)+i\epsilon)}+\partial_y\mathcal{B}\,\right\|_{X^{1,2}_{k,y_0,\epsilon}}\leq C \big[ \log{\langle k\rangle}\big]^4\|f\|_{X^3_{k,y_0,\epsilon}}.
\end{equation}
In the above we have set the boundary terms $\mathcal{B}$ as
\begin{equation}\label{Bt}
\mathcal{B}:=g(1)\frac{\sinh{(ky)}}{|b'(1)|^2\sinh{k}}\log{(b(1)-b(y_0)+i\epsilon)}+g(0)\frac{\sinh{(k(1-y))}}{|b'(0)|^2\sinh{k}}\log{(b(0)-b(y_0)+i\epsilon)}.
\end{equation}

 As a corollary of \eqref{X17}, for $f$ with $f(0)=f(1)=0$  we also have
\begin{equation}\label{X17.1}
|k|\left\|T_{k,y_0,\epsilon}f\right\|_{X^2_{k,y_0,\epsilon}+Y^1_{k,y_0,\epsilon}}+\left\|\partial_yT_{k,y_0,\epsilon}f\right\|_{X^3_{k,y_0,\epsilon}+X^{1,2}_{k,y_0,\epsilon}}\leq C\big[\log{\langle k\rangle}\big]^4\left\|f\right\|_{X^3_{k,y_0,\epsilon}}.
\end{equation}

\end{lemma}

\begin{proof}
We normalize so that $\|f\|_{X^3}:=1$. Clearly \eqref{X17.1} is a consequence of \eqref{X17}-\eqref{X17.0} and the definitions, since $\mathcal{B}=0$ in this case.

 We begin with the proof of \eqref{X17}.
 We first make the simple observation that for $f(y)=g(y)/\big[b(y)-b(y_0)+i\epsilon\big]$ with $\|g\|_{Z^1}\lesssim|k|^{-1}$, the following identity holds
\begin{equation}\label{X18}
\begin{split}
Tf(y)=&\int_0^1\partial_zG_k(y,z) \frac{g(z)}{b'(z)\big[b(z)-b(y_0)+i\epsilon\big]}\,dz+\int_0^1\partial_z\left[\frac{g(z)}{b'(z)}\right]\frac{G_k(y,z) }{b(z)-b(y_0)+i\epsilon}\,dz\\
        =&-\mathcal{B}+\frac{g(y)}{|b'(y)|^2}\log{(b(y)-b(y_0)+i\epsilon)}+\int_0^1\partial_z\left[\frac{g(z)}{b'(z)}\right]\frac{G_k(y,z) }{b(z)-b(y_0)+i\epsilon}\,dz\\
          &\qquad-k^2\int_0^1G_k(y,z)\frac{g(z)\log{(b(z)-b(y_0)+i\epsilon)}}{|b'(z)|^2}\,dz\\
          &\qquad-\int_0^1\partial_zG_k(y,z)\partial_z\left[ \frac{g(z)}{|b'(z)|^2}\right] \log{(b(z)-b(y_0)+i\epsilon)}\,dz:=-\mathcal{B}+\mathcal{T}_1+\mathcal{T}_2+\mathcal{T}_3+\mathcal{T}_4.
         \end{split}
\end{equation}
It is sufficient to bound $\mathcal{T}_i, i\in\{2,3,4\}$. We will show that $\mathcal{T}_i, i\in\{2,3,4\}$ satisfy
\begin{equation}\label{X19}
\|\mathcal{T}_i\|_{Y^1}\lesssim |k|^{-1}\big[\log{\langle k\rangle}\big]^4, \qquad {\rm for}\,\,i\in\{2,3,4\}.
\end{equation}
 The most difficult term is
\begin{equation}\label{X20}
\mathcal{T}_2':=\int_0^1\frac{g'(z)}{b'(z)}\frac{G_k(y,z) }{b(z)-b(y_0)+i\epsilon}\,dz=\int_0^1\frac{g'(z)}{|b'(z)|^2}G_k(y,z) \partial_z\log{(b(z)-b(y_0)+i\epsilon)}\,dz
\end{equation} 
 from $\mathcal{T}_2$, which involves additional cancellations. Since $\|g\|_{Z^1}\leq |k|^{-1}$, we can write
 \begin{equation}\label{X20.1}
 g'(y)=h_1(y)\log{(b(y)-b(y_0)+i\epsilon)}+h_2(y),\qquad {\rm with}\,\,\|h_1\|_{Y^1}+\|h_2\|_{Y^{1,2}}\lesssim1.
 \end{equation}
 We use integration by parts and obtain that
 \begin{equation}\label{X21}
 \begin{split}
 -\mathcal{T}_2'=&\frac{1}{2}\int_0^1\partial_z\left[\frac{h_1(z)}{b'(z)^2}\right]G_k(y,z)\big[\log{(b(z)-b(y_0)+i\epsilon)}\big]^2\,dz\\
                          &\qquad+\frac{1}{2}\int_0^1\frac{h_1(z)}{b'(z)^2}\partial_zG_k(y,z)\big[\log{(b(z)-b(y_0)+i\epsilon)}\big]^2\,dz\\
                          &+\int_0^1\partial_z\left[\frac{h_2(z)}{b'(z)^2}\right]G_k(y,z)\log{(b(z)-b(y_0)+i\epsilon)}\,dz\\
                          &\qquad+\int_0^1\frac{h_2(z)}{b'(z)^2}\partial_zG_k(y,z)\log{(b(z)-b(y_0)+i\epsilon)}\,dz
 \end{split}
 \end{equation}
 The desired bounds \eqref{X19} follow from \eqref{X21} and the bounds \eqref{X20.1}.
 
 Finally, the proof of \eqref{X17.0} follows from taking derivatives in $y$ in \eqref{X18} and \eqref{X21}, and estimating the resulting terms in a straightforward fashion. We omit the routine details.
\end{proof}

\section{The limiting absorption principle}
Define for each $\epsilon\in[-1/4,1/4]\backslash\{0\}, k\in\mathbb{Z}\backslash\{0\},y_0\in[0,1]$ the operator 
\begin{equation}\label{Sep}
S_{k,y_0,\epsilon}f(y):=-\int_0^1G_k(y,z)\frac{b''(z)f(z)}{b(z)-(y_0)+i\epsilon}\,dz.
\end{equation}

The following lemma is the main tool we shall use to obtain estimates on $\psi^{\iota}_{k,\epsilon}(y,y_0)$.
\begin{lemma}\label{X22}
For each $\epsilon\in[-1/4,1/4]\backslash\{0\}, k\in\mathbb{Z}\backslash\{0\},y_0\in[0,1]$ and $m\in\{1,2,3,4,5\}$, let $V_{k,y_0,\epsilon}\in\{Y^1_{k,y_0,\epsilon}, X^2_{k,y_0,\epsilon}+Y^1_{k,y_0,\epsilon}, X^{1,m}_{k,y_0,\epsilon}\}$. There exists $\delta>0$ such that for sufficiently small $\epsilon\neq0$
\begin{equation}\label{X23}
\big\|(I-S_{k,y_0,\epsilon})f\big\|_{V_{k,y_0,\epsilon}}\ge \delta \big\|f\big\|_{V_{k,y_0,\epsilon}},
\end{equation}
for any $k\in\mathbb{Z}\backslash\{0\},y_0\in[0,1]$.
\end{lemma}

\begin{proof}
We consider only the case when $V_{k,y_0,\epsilon}=X^{1,5}_{k,y_0,\epsilon}$, the other two cases are similar and the proof follows the same line of argument.

We argue by contradiction. Suppose the lemma does not hold, then there exist a sequence $\epsilon_j\neq 0, \epsilon_j\to 0$, $y_j\in [0,1], y_j\to y_0$, $k_j\in \mathbb{Z}\backslash\{0\}$, $k_j\to k_0\in(\mathbb{Z}\backslash\{0\})\cup \{\pm\infty\}$, and functions $\psi_j$ such that
\begin{equation}\label{X24}
\|\psi_j\|_{X^{1,5}_{k_j,y_j,\epsilon_j}}=1,
\end{equation} 
and
\begin{equation}\label{X25}
\left\|\psi_j-S_{k_j,y_j,\epsilon_j}\psi_j\right\|_{X^{1,5}_{k_j,y_j,\epsilon_j}}\to 0.
\end{equation}
To simply notations, we set
\begin{equation}\label{X26}
X^{1,5}_j:=X^{1,5}_{k_j,y_j,\epsilon_j}, \qquad Y^{1,5}_j:=Y^{1,5}_{k_j,y_j,\epsilon_j},\qquad Z^{1,5}_j:=Z^{1,5}_{k_j,y_j,\epsilon_j}\qquad S_j:=S_{k_j,y_j,\epsilon_j}.
\end{equation}
Write
\begin{equation}\label{X27}
\psi_j=S_j\psi_j+r_j,
\end{equation}
where, by \eqref{X25},
\begin{equation}\label{X28}
\lim_{j\to\infty}\|r_j\|_{X^{1,5}_j}=0.
\end{equation}

{\bf Step 1} In this step we obtain improved regularity property for $\psi_j$. From \eqref{X27} we get
\begin{equation}\label{X29}
\psi_j=S_j(S_j\psi_j)+S_jr_j+r_j.
\end{equation}
By Lemma \ref{X11}-\ref{bX17}, and \eqref{X24} and \eqref{X28}, we obtain that
\begin{equation}\label{X30}
\|S_jr_j\|_{X^{1,5}_j}+\|r_j\|_{X_j^{1,5}}=o_j(1)
\end{equation}
and
\begin{equation}\label{X31}
\|S_j(S_j\psi_j)\|_{Z^{1,5}_j}\lesssim |k_j|^{-1}.
\end{equation}
\eqref{X29}-\eqref{X31} and \eqref{X24} imply that $|k_j|\lesssim1$.
Hence we can assume (by passing to a subsequence) that
\begin{equation}\label{X34}
k_j=k_0\in\mathbb{Z}\backslash\{0\} \qquad{\rm for\,\,all}\,\,j.
\end{equation}
Set
\begin{equation}\label{X35}
\psi_{2j}:=S_j^2\psi_j\qquad{\rm and}\qquad r_{1j}=S_jr_j,
\end{equation}
one has
\begin{equation}\label{X36}
\left\|\psi_{2j}\right\|_{Z^{1,5}_j}\lesssim1, \qquad \left\|r_{1j}\right\|_{X^{1,5}_j}=o_j(1),
\end{equation}
and
\begin{equation}\label{X37}
\psi_j=S_j\psi_j+r_j=\psi_{2j}+r_{1j}+r_j.
\end{equation}

{\bf Step 2} In this step we pass $j\to\infty$. We first show that
\begin{equation}\label{X39}
\limsup_{j\to\infty}\left\|\psi_{2j}\right\|_{H^1}\gtrsim1.
\end{equation}
Suppose on the contrary one has
\begin{equation}\label{X40}
\lim_{j\to\infty}\left\|\psi_{2j}\right\|_{H^1}=0.
\end{equation}
Using formula similar to \eqref{bX4}, we have
\begin{equation}\label{X41}
\begin{split}
\partial_y(S_j\psi_{2j})(y)=&\frac{b''(y)\psi_{2j}(y)}{b'(y)}\log{(b(y)-b(y_j)+i\epsilon_j)}\\
                    &+\int_0^1G_{k_0}'(y,z)\frac{b''(z)\psi_{2j}(z)}{b'(z)}\log{(b(z)-b(y_j)+i\epsilon_j)}\,dz\\
                   &+\int_0^1\partial_yG_{k_0}(y,z)\partial_z\left[\frac{b''(z)\psi_{2j}(z)}{b'(z)}\right]\log{(b(z)-b(y_j)+i\epsilon_j)}\,dz.
\end{split}
\end{equation}
By H\"older inequality, we obtain from \eqref{X40} that
\begin{equation}\label{X42}
\lim_{j\to\infty}\left\|\frac{\partial_y(S_j\psi_{2j})(y)}{\log{(b(y)-b(y_j)+i\epsilon_j)}-\vartheta^{-1}}\right\|_{L^{\infty}}=0,
\end{equation}
and thus 
\begin{equation}\label{X42.1}
\|S_j\psi_{2j}\|_{Y_j^{1,5}}=o_j(1).
\end{equation}
Using the identity
$$\psi_j=S_j\psi_{2j}+S_jr_{1j}+S_jr_j+r_j$$
and 
$$\|S_jr_{1j}+S_jr_j+r_j\|_{X^{1,5}_j}=o_j(1),$$
we obtain a contradiction with \eqref{X24} from \eqref{X42.1}, which finishes the proof of \eqref{X39}.

Using the bound \eqref{X36} and \eqref{X39}, by passing to a subsequence, there exists a $\psi_0\in H_0^1$ such that
\begin{equation}\label{X43}
\lim_{j\to\infty}\psi_{2j}=\psi_0\not\equiv 0, \qquad {\rm in}\,\,H^1.
\end{equation}
Denote
\begin{equation}\label{X43.1}
\phi_j:=\psi_{2j}.
\end{equation}
Then
\begin{equation}\label{X43.2}
\|\phi_j\|_{Z^{1,5}_j}\lesssim 1\qquad{\rm and}\qquad \lim_{j\to\infty}\|\phi_j-\psi_0\|_{H^1}=0.
\end{equation}
From \eqref{X35}-\eqref{X37} and recalling the definitions \eqref{Sep}, we conclude that the nontrivial function $\psi_0\in H_0^1$ satisfies the following equation
\begin{equation}\label{X44}
\psi_0+\lim_{j\to\infty}\int_0^1G_{k_0}(y,z)\frac{b''(z)\phi_j(z)}{b(z)-b(y_j)+i\epsilon_j}\,dz=0.
\end{equation}

{\bf Step 3} Finally we obtain a contradiction using \eqref{X43.2}-\eqref{X44} and the absence of embedded eigenvalues assumption (A). In view of \eqref{X43.2} and the formula
\begin{equation}\label{X45}
\int_0^1G_{k_0}(y,z)\frac{b''(z)\phi_j(z)}{b(z)-b(y_j)+i\epsilon_j}\,dz=-\int_0^1\partial_z\left[G_{k_0}(y,z)\frac{b''(z)\phi_j(z)}{b'(z)}\right]\log{(b(z)-b(y_j)+i\epsilon_j)}\,dz,
\end{equation}
it follows from \eqref{X43.2} and \eqref{X44} that
\begin{equation}\label{X46}
\psi_0+\lim_{j\to\infty}\int_0^1G_{k_0}(y,z)\frac{b''(z)\psi_0(z)}{b(z)-b(y_0)+i\epsilon_j}\,dz=0.
\end{equation}
We see from \eqref{X46} that, in the sense of distributions,
\begin{equation}\label{X46.1}
-\frac{d^2}{dy^2}\psi_0+k^2\psi_0+\lim_{j\to\infty}\frac{b(y)-b(y_0)}{(b(y)-b(y_0))^2+\epsilon_j^2}b''(y)\psi_0-iC_{y_0}b''(y_0)\psi_0(y_0)\delta(y-y_0)=0,
\end{equation}
with a constant $C_{y_0}>0$.
Multiplying \eqref{X46.1} with $\overline{\psi_0(y)}$, integrating over $[0,1]$ and taking the imaginary part, we see that
\begin{equation}\label{X47}
b''(y_0)\psi_0(y_0)=0.
\end{equation}
Setting 
\begin{equation}\label{X48}
f_{j}:=\frac{b''(y)\psi_0(y)}{b(y)-b(y_0)+i\epsilon_j},
\end{equation}
we obtain from \eqref{X47} and \eqref{X46} that
\begin{equation}\label{X49}
\lim_{j\to\infty}f_j=f\in L^2\qquad {\rm and}\qquad (b(y)-b(y_0))f(y)+b''(y)\int_0^1G_{k_0}(y,z)f(z)dz=0.
\end{equation}
\eqref{X49} is a contradiction to the assumption (A) in subsection \ref{idea} on the absence of embedded eigenvalues. The lemma is now proved.
\end{proof}

\section{Regularity of the spectral measure}
We now study the regularity of $\psi^{\iota}_{k,\epsilon}(y,y_0)$ with $\iota\in\{\pm \}$ and $y, y_0\in[0,1]$, in the limit $\epsilon\to0$. We begin with the equation \eqref{F7}, which can be reformulated as
\begin{equation}\label{F8}
\psi^{\iota}_{k,\epsilon}(y,y_0)+\int_0^1G_k(y,z)\frac{b''(z)\psi^{\iota}_{k,\epsilon}(z,y_0)}{b(z)-b(y_0)+i\iota\epsilon}\,dz=\int_0^1G_k(y,z)\frac{\omega^k_0(z)}{b(z)-b(y_0)+i\iota\epsilon}\,dz.
\end{equation}
Our goal is to obtain estimates on $\psi^{\iota}_{k,\epsilon}(y,y_0)$, $\partial_{y_0}\psi^{\iota}_{k,\epsilon}(y,y_0)$, and $\partial^2_{y_0}\psi^{\iota}_{k,\epsilon}(y,y_0)$.
Choose $\epsilon_0>0$ sufficiently small so that \eqref{X23} holds. Denote
\begin{equation}\label{L01.1}
\Sigma:=\{(k,y_0,\epsilon): k\in\mathbb{Z}\backslash\{0\}, y_0\in[0,1]\,\, {\rm and} \,\,\epsilon\in(-\epsilon_0,\epsilon_0)\backslash\{0\}\}.
\end{equation}

To obtain good dependence of various constants on the parameter $k\in\mathbb{Z}\backslash\{0\}$, we define
\begin{equation}\label{Hk}
\big\|\omega_0^k\big\|_{H^3_k(0,1)}:=\sum_{0\leq\alpha\leq3}|k|^{3-\alpha}\big\|\partial_{y}^{\alpha}\omega_0^k\big\|_{L^2(0,1)}.
\end{equation}
We have the elementary inequality
\begin{equation}\label{Hk1}
\sum_{0\leq\alpha\leq2}|k|^{2-\alpha}\big\|\partial_y^{\alpha}\omega_0^k\big\|_{L^{\infty}(0,1)}\lesssim\big\|\omega_0^k\big\|_{H^3_k(0,1)}.
\end{equation}
Throughout this section, we normalize $\big\|\omega_0^k\big\|_{H^3_k(0,1)}=1$.

We first prove a technical lemma needed below.
\begin{lemma}\label{L0.3}
For $\sigma\in\{1,2,3\}$, $\iota\in\{\pm\}$, $(k,y_0,\epsilon)\in\Sigma$, set
\begin{equation}\label{Fs}
\begin{split}
F^{\sigma\iota}_{k,y_0,\epsilon}:&=\int_0^1G_k(y,z)\frac{\omega_0^k(z)}{(b(z)-b(y_0)+i\iota\epsilon)^\sigma}\,dz.
\end{split}
\end{equation}
Then
\begin{equation}\label{bF1}
\sup_{\iota=\pm}\left\|F^{1\iota}_{k,y_0,\epsilon}\right\|_{Z^1_{k,y_0,\iota\epsilon}}\lesssim|k|^{-3}\big[\log{\langle k\rangle}\big]^3;
\end{equation}
\begin{equation}\label{bF2}
\sup_{\iota=\pm}\left\|F^{2\iota}_{k,y_0,\epsilon}-\frac{\omega_0^k(y)}{|b'(y)|^2}\log{(b(y)-b(y_0)+i\iota\epsilon)}+\mathcal{B}^{\iota}_{k1}\right\|_{Y^1_{k,y_0,\iota\epsilon}}\lesssim|k|^{-2}\big[\log{\langle k\rangle}\big]^4
\end{equation}
with
\begin{equation}\label{bF2.0}
\sup_{\iota=\pm1}\left\|\partial_yF^{2\iota}_{k,y_0,\epsilon}-\frac{\omega_0^k(y)}{b'(y)(b(y)-b(y_0)+i\iota\epsilon)}+\partial_y\mathcal{B}^{\iota}_{k1}\right\|_{X^{1,2}_{k,y_0,\iota\epsilon}}\lesssim|k|^{-1}\big[\log{\langle k\rangle}\big]^4;
\end{equation}
and
\begin{equation}\label{bF3}
\sup_{\iota=\pm1}\left\|F^{3\iota}_{k,y_0,\iota\epsilon}+\frac{\omega_0^k(y)}{2|b'(y)|^2(b(y)-b(y_0)+i\iota\epsilon)}+\mathcal{B}^{\iota}_{k2}\right\|_{X^{1,2}_{k,y_0,\iota\epsilon}}\lesssim|k|^{-1}\big[\log{\langle k\rangle}\big]^4.
\end{equation}
In the above the boundary terms $\mathcal{B}^{\iota}_{k1}, \mathcal{B}^{\iota}_{k2}$ are defined as
\begin{equation}\label{bF2.1}
\mathcal{B}^{\iota}_{k1}:=\omega^k_0(1)\frac{\sinh{(ky)}}{|b'(1)|^2\sinh{k}}\log{(b(1)-b(y_0)+i\iota\epsilon)}+\omega^k_0(0)\frac{\sinh{(k(1-y))}}{|b'(0)|^2\sinh{k}}\log{(b(0)-b(y_0)+i\iota\epsilon)},
\end{equation}
\begin{equation}\label{bF3.1}
\begin{split}
\mathcal{B}^{\iota}_{k2}:=&-\frac{1}{2}\omega^k_0(1)\frac{\sinh{(ky)}}{|b'(1)|^2\sinh{k}}\frac{1}{[b(1)-b(y_0)+i\iota\epsilon]}-\frac{1}{2}\omega^k_0(0)\frac{\sinh{(k(1-y))}}{|b'(0)|^2\sinh{k}}\frac{1}{[b(0)-b(y_0)+i\iota\epsilon]}\\
                                        &+\left[(-3/2)\omega_0^k(1)\frac{b''(1)}{(b'(1))^4}+\frac{d}{dy}\omega_0^k(1)\frac{1}{(b'(1))^3}\right]\frac{\sinh{(ky)}}{\sinh{k}}\log{(b(1)-b(y_0)+i\iota\epsilon)}\\
                                        &+\left[(-3/2)\omega_0^k(1)\frac{b''(0)}{(b'(0))^4}+\frac{d}{dy}\omega_0^k(0)\frac{1}{(b'(0))^3}\right]\frac{\sinh{(k(1-y))}}{\sinh{k}}\log{(b(0)-b(y_0)+i\iota\epsilon)}.
\end{split}
\end{equation}
\end{lemma}

\begin{proof}
\eqref{bF1} follows from Lemma \ref{bX1}. The bounds \eqref{bF2}-\eqref{bF2.0} follow from Lemma \ref{bX17}.

To prove \eqref{bF3}, we follow similar integration by parts argument, as in the proof of Lemma \ref{bX17}, and get
\begin{equation}\label{fF4}
\begin{split}
F^{3\iota}_{k,y_0,\epsilon}=&\frac{1}{2}\int_0^1\partial_z\left\{\frac{1}{b'(z)}\partial_z\left[\frac{G_k(y,z)\omega_0^k(z)}{b'(z)}\right]\right\}\frac{1}{b(z)-b(y_0)+i\iota\epsilon}\,dz\\
                                   &-\frac{1}{2}\frac{\partial_zG_k(y,z)\omega_0^k(z)}{|b'(z)|^2}\frac{1}{b(z)-b(y_0)+i\iota\epsilon}\,\bigg|_{z=0}^1.
\end{split}
\end{equation}
Set
\begin{equation}\label{fF4.0}
\mathcal{B}_{k2}^{\iota\ast}:=-\frac{\omega^k_0(1)}{2}\frac{\sinh{(ky)}}{|b'(1)|^2\sinh{k}}\frac{1}{[b(1)-b(y_0)+i\iota\epsilon]}-\frac{\omega^k_0(0)}{2}\frac{\sinh{(k(1-y))}}{|b'(0)|^2\sinh{k}}\frac{1}{[b(0)-b(y_0)+i\iota\epsilon]}.
\end{equation}
From \eqref{fF4}, we obtain that
 \begin{equation}\label{fF5}
 \begin{split}
 F^3_{k,y_0,\epsilon}+\mathcal{B}^{\iota\ast}_{k2}=&-\frac{1}{2}\frac{\omega_0^k(y)}{|b'(y)|^2}\frac{1}{b(y)-b(y_0)+i\iota\epsilon}+\int_0^1\frac{\partial_zG_k(y,z)\partial_z\omega_0^k(z)}{|b'(z)|^2}\frac{1}{b(z)-b(y_0)+i\iota\epsilon}\,dz\\
 &-\frac{3}{2}\int_0^1\frac{b''(z)}{(b'(z))^3}\frac{\omega_0^k(z)\partial_zG_k(y,z)}{b(z)-b(y_0)+i\iota\epsilon}dz-\frac{3}{2}\int_0^1\frac{b''(z)}{(b'(z))^3}\frac{\partial_z\omega_0^k(z)G_k(y,z)}{b(z)-b(y_0)+i\iota\epsilon}dz\\
 &+\frac{1}{2}\int_0^1\frac{G_k(y,z)\partial_z^2\omega_0^k(z)}{|b'(z)|^2(b(z)-b(y_0)+i\iota\epsilon)}\,dz+\frac{k^2}{2}\int_0^1\frac{\omega_0^k(z)}{|b'(z)|^2}\frac{G_k(z,y_0)}{b(z)-b(y_0)+i\iota\epsilon}dz\\
 &+\frac{1}{2}\int_0^1\left[\frac{|b''(z)|^2}{|b'(z)|^4}+\frac{1}{b'(z)}\partial_z^2\left(\frac{1}{b'(z)}\right)\right]\frac{G_k(y,z)\omega_0^k(z)}{b(z)-b(y_0)+i\iota\epsilon}dz.
 \end{split}
 \end{equation}
 The first term on the right hand side is part of \eqref{bF3}. For the other terms,
 upon writing 
 $$\frac{1}{b(z)-b(y_0)+i\iota\epsilon}=\frac{1}{b'(z)}\partial_z\log{(b(z)-b(y_0)+i\iota\epsilon)}$$
we can integrate by parts in $z$ again, and, with the boundary terms collected, the resulting terms can be estimated in a straightforward fashion, using \eqref{Gk1.1} and \eqref{Gk3.1}, which completes the proof of \eqref{bF3}.
\end{proof}

We turn now to the property of the generalized eigenfunctions, and begin with the property of $\psi^{\iota}_{k,\epsilon}(y,y_0)$.
\begin{lemma}\label{L0.1}
Let $\psi^{\iota}_{k,\epsilon}(y,y_0)$ with $\iota\in\{\pm \}$ be defined as above.  Recall the definition \eqref{L01.1}. Then
\begin{equation}\label{L01.2}
\begin{split}
\sup_{(k,y_0,\epsilon)\in\Sigma}\big[\log{\langle k\rangle}\big]^{-4} \bigg[|k|^3\big\|\psi^{\iota}_{k,\epsilon}(\cdot,y_0)\big\|_{Z^1_{k,y_0,\iota\epsilon}}\bigg]\lesssim 1.\end{split}
\end{equation}

\end{lemma}

\begin{proof}
We normalize so that $\|\omega_0^k\|_{H^3}=1$. Recall the definition \eqref{Sep}.
Denote
\begin{equation}\label{L2.1}
S^{\iota}_{k,y_0,\epsilon}:=S_{k,y_0,\iota\epsilon},\qquad{\rm for}\,\,\iota\in\{\pm\}.
\end{equation}
We first note that \eqref{F8} can be written in the following abstract form
\begin{equation}\label{L3}
(I-S^{\iota}_{k,y_0,\epsilon})\psi^{\iota}_{k,\epsilon}(y,y_0)=\int_0^1G_k(y,z)\frac{\omega^k_0(z)}{b(z)-b(y_0)+i\iota\epsilon}\,dz.
\end{equation}

Using Lemma \ref{bX1} and Lemma \ref{X22}, and the bounds \eqref{bF1}, we can conclude that \eqref{L3} 
\begin{equation}\label{L4}
\big\|\psi^{\iota}_{k,\epsilon}(\cdot,y_0)\big\|_{Y^1_{k,y_0,\iota\epsilon}}\lesssim |k|^{-3}\big[\log{\langle k\rangle}\big]^3. 
\end{equation}
Using \eqref{F8} and Lemma \ref{bX1}, we can upgrade \eqref{L4} and obtain that
\begin{equation}\label{L5}
\big\|\psi^{\iota}_{k,\epsilon}(\cdot,y_0)\big\|_{Z^1_{k,y_0,\iota\epsilon}}\lesssim |k|^{-3}\big[\log{\langle k\rangle}\big]^3,
\end{equation}
which completes the proof of \eqref{L01.2} for $\psi^{\iota}_{k,\epsilon}(\cdot,y_0)$. 

\end{proof}

We next turn to the property of $\partial_{y_0}\psi^{\iota}_{k,\epsilon}$.
\begin{lemma}\label{L0.2}
$\partial_{y_0}\psi^{\iota}_{k,\epsilon}(y,y_0)$ and $\partial_{y_0}\partial_y\psi^{\iota}_{k,\epsilon}(y,y_0)$ satisfy the following decomposition
\begin{equation}\label{L0.21}
\begin{split}
\partial_{y_0}\psi^{\iota}_{k,\epsilon}(y,y_0)=&\bigg[\frac{b'(y_0)\omega^k_0(y)}{|b'(y)|^2}-\frac{b'(y_0)b''(y)\psi^{\iota}_{k,\epsilon}(y,y_0)}{|b'(y)|^2}\bigg]\log{(b(y)-b(y_0)+i\iota\epsilon)}\\  
&\\
         &+b'(y_0)\sum_{\sigma=0,1}\omega_0^k(\sigma)\Psi^{\iota}_{\sigma,k,y_0,\epsilon}(y)\log{(b(\sigma)-b(y_0)+i\iota\epsilon)}+\mathcal{R}^{\iota}_{\sigma,k,y_0,\epsilon}(y),
\end{split}
\end{equation}
and
\begin{equation}\label{L0.22}
\begin{split}
\partial_y\partial_{y_0}\psi^{\iota}_{k,\epsilon}(y,y_0)&=\left[\frac{b'(y_0)\omega^k_0(y)}{b'(y)}-\frac{b'(y_0)b''(y)\psi^{\iota}_{k,\epsilon}(y,y_0)}{b'(y)}\right]\frac{1}{b(y)-b(y_0)+i\iota\epsilon}\\
&\\
                                              &+b'(y_0)\sum_{\sigma=0,1}\omega_0^k(\sigma)\partial_y\Psi^{\iota}_{\sigma,k,y_0,\epsilon}(y)\log{(b(\sigma)-b(y_0)+i\iota\epsilon)}+\widetilde{\mathcal{R}}^{\iota}_{k,y_0,\epsilon}(y).
\end{split}
\end{equation}
In the above
\begin{equation}\label{L0.211}
\left\|\mathcal{R}^{\iota}_{k,y_0,\epsilon}\right\|_{Y^1_{k,y_0,\epsilon}}\lesssim |k|^{-2}\big[\log{\langle k\rangle}\big]^4,\qquad
\left\|\widetilde{\mathcal{R}}^{\iota}_{k,y_0,\epsilon}\right\|_{X^{1,2}_{k,y_0,\iota\epsilon}}\lesssim|k|^{-1} \big[\log{\langle k\rangle}\big]^4;
\end{equation}
the functions $\Psi^{\iota}_{\sigma,k,y_0,\epsilon}$ for $\sigma\in\{0,1\}$, $\iota\in\{\pm\}$ and $(k,y_0,\epsilon)\in\Sigma$ satisfy
\begin{equation}\label{L0.200}
\left\|\Psi^{\iota}_{\sigma,k,y_0,\epsilon}\right\|_{Z^{1}_{k,y_0,\iota\epsilon}}\lesssim\big[\log{\langle k\rangle}\big]^4,\qquad \qquad\lim_{\epsilon\to0}\sum_{\alpha\in\{0,1\}}\left|\Psi^{+}_{\sigma,k,\alpha,\epsilon}(y)- \Psi^{-}_{\sigma,k,\alpha,\epsilon}(y)\right|=0,
\end{equation}
and
\begin{equation}\label{L0.20001}
\left|\partial_{y_0,\epsilon}\Psi^{\iota}_{\sigma,k,y_0,\epsilon}(y)\right|\lesssim \big[\log{\langle k\rangle}\big]^4\bigg[\sum_{\alpha\in\{0,1,y\}}\big|\log{(b(\alpha)-b(y_0)+i\iota\epsilon)}\big|+1\bigg].
\end{equation}
We also record the following property of $\psi^{\iota}_{k,\epsilon}(y,y_0)$,
\begin{equation}\label{L0.20002}
\lim_{\epsilon\to0}\Big[\psi^{+}_{k,\epsilon}(y,y_0)-\psi^{-}_{k,\epsilon}(y,y_0)\Big]=0, \qquad {\rm for\,\,}y_0\in\{0,1\},
\end{equation}
\begin{equation}\label{L0.20003}
\left|\partial_{\epsilon}\psi^{\iota}_{k,\epsilon}(y,y_0)\right|\lesssim |k|^{-2}\big[\log{\langle k\rangle}\big]^4\bigg[\sum_{\alpha\in\{0,1,y\}}\big|\log{(b(\alpha)-b(y_0)+i\iota\epsilon)}\big|+1\bigg].
\end{equation}
\end{lemma}

\begin{proof}
Taking one derivative in $y_0$, we obtain from \eqref{F8},
\begin{equation}\label{L6}
\begin{split}
\partial_{y_0}\psi^{\iota}_{k,\epsilon}(y,y_0)+\int_0^1G_k(y,z)\frac{b''(z)\partial_{y_0}\psi^{\iota}_{k,\iota\epsilon}(z,y_0)}{b(z)-b(y_0)+i\iota\epsilon}\,dz=&\int_0^1G_k(y,z)\frac{b'(y_0)\omega_0^k(z)}{(b(z)-b(y_0)+i\iota\epsilon)^2}\,dz\\
&-\int_0^1G_k(y,z)\frac{b'(y_0)b''(z)\psi^{\iota}_{k,\iota\epsilon}(z,y_0)}{(b(z)-b(y_0)+i\iota\epsilon)^2}\,dz\\
:=&\mathcal{G}^{\iota}_{k,y_0,\epsilon}(y).
\end{split}
\end{equation}
For $\iota\in\{\pm\}$, $k\in\mathbb{Z}\backslash\{0\}$, $y_0\in[0,1]$, $\epsilon\in[-1/4,1/4]\backslash\{0\}$,
\begin{equation}\label{L6.1}
\begin{split}
\mathcal{B}^{\iota}_{k,y_0,\epsilon}(y):=&\,\omega_0^k(1)b'(y_0)\frac{\sinh{(ky)}}{|b'(1)|^2\sinh{k}}\log{(b(1)-b(y_0)+i\iota\epsilon)}\\
                                         &+\omega_0^k(0)b'(y_0)\frac{\sinh{k(1-y)}}{|b'(0)|^2\sinh{k}}\log{(b(0)-b(y_0)+i\iota\epsilon)};
\end{split}
\end{equation}
Using \eqref{bF2}, Lemma \ref{bX17} and \eqref{L5}, we conclude that
\begin{equation}\label{L7}
\|\mathcal{G}^{\iota}_{k,y_0,\epsilon}+\mathcal{B}^{\iota}_{k,y_0,\epsilon}\|_{X^2_{k,y_0,\iota\epsilon}+Y^1_{k,y_0,\iota\epsilon}}\lesssim|k|^{-2}\big[\log{\langle k\rangle}\big]^4.
\end{equation}
Define $\Psi^{\iota'}_{\sigma,k,y_0,\epsilon}$ for $\iota\in\{\pm\}$, $\sigma\in\{0,1\}$, $(k,y_0,\epsilon)\in\Sigma$ as the solution to
\begin{equation}\label{L7.1}
(I-S_{k,y_0,\epsilon}^{\iota})\Psi^{\iota'}_{1,k,y_0,\epsilon}=-S_{k,y_0,\epsilon}^{\iota}\left[\frac{\sinh{(ky)}}{|b'(1)|^2\sinh{k}}\right],
\end{equation}
\begin{equation}\label{L7.2}
(I-S_{k,y_0,\epsilon}^{\iota})\Psi^{\iota'}_{0,k,y_0,\epsilon}=-S_{k,y_0,\epsilon}^{\iota}\left[\frac{\sinh{k(1-y)}}{|b'(0)|^2\sinh{k}}\right].
\end{equation}
Now set for $\iota\in\{\pm\}$, $(k,y_0,\epsilon)\in\Sigma$,
\begin{equation}\label{L7.21}
\begin{split}
\Psi^{\iota}_{1,k,y_0,\epsilon}=-\frac{\sinh{(ky)}}{|b'(1)|^2\sinh{k}}+\Psi^{\iota'}_{1,k,y_0,\epsilon},\qquad\Psi^{\iota}_{0,k,y_0,\epsilon}=-\frac{\sinh{k(1-y)}}{|b'(0)|^2\sinh{k}}+\Psi^{\iota'}_{0,k,y_0,\epsilon}.
\end{split}
\end{equation}
By Lemma \ref{bX1} and Lemma \ref{X22}, we have the following bounds for all $\iota\in\{\pm\}$, $k\in\mathbb{Z}\backslash\{0\}$,
\begin{equation}\label{L7.3}
\sum_{\sigma=0,1}\left\|\Psi^{\iota}_{\sigma,k,y_0,\epsilon}\right\|_{Z^1_{k,y_0,\iota\epsilon}}\lesssim\big[\log{\langle k\rangle}\big]^3.
\end{equation}

 It is clear that
$$\sum_{\sigma\in\{0,1\}}b'(y_0)\omega_0^k(\sigma)\Psi^{\iota}_{\sigma,k,y_0,\epsilon}(y)\log{(b(\sigma)-b(y_0)+i\iota)}$$
solves \eqref{L6} with the right hand side $-\mathcal{B}^{\iota}_{\sigma,k,y_0,\epsilon}$.

To prove \eqref{L0.21}, we only need to consider \eqref{L6} with the right hand side $\mathcal{G}^{\iota}_{k,y_0,\epsilon}+\mathcal{B}^{\iota}_{k,y_0,\epsilon}$. We use Lemma \ref{X22} for the norms $X^2_{k,y_0,\iota\epsilon}+Y^1_{k,y_0,\iota\epsilon}$.   In view of Lemma \ref{bX17} and \eqref{L7}, combining with the bounds \eqref{L7.3}, the decomposition \eqref{L0.21} is now established.

The decomposition \eqref{L0.22} follows from Lemma \ref{bX17}, when considering equation \eqref{L7} with right hand side $\mathcal{G}^{\iota}_{k,y_0,\epsilon}+\mathcal{B}^{\iota}_{k,y_0,\epsilon}$, after the boundary terms are taken out.

\eqref{L0.20001} follows from taking derivatives in $y_0,\epsilon$ in the equations \eqref{L7.1}-\eqref{L7.2}, and apply the same argument for the proof of the decompositions \eqref{L0.21}-\eqref{L0.22}. \eqref{L0.20003} follows similarly.
\end{proof}

To obtain better control on the generalized eigenfunctions, we need the following lemma.
\begin{lemma}\label{P0}
For $\iota\in\{\pm\}, \sigma$, $(k,y_0,\epsilon)\in\Sigma$, set
\begin{equation}\label{P1}
F^{4\iota}_{k,y_0,\epsilon}(y):=\int_0^1G_k(y,z)b''(z)\frac{\psi^{\iota}_{k,\epsilon}(z,y_0)}{(b(z)-b(y_0)+i\iota\epsilon)^3}dz,
\end{equation}
Then there exist functions $\Upsilon^{\iota ij}_{\sigma,\tau,k,y_0,\epsilon}(y)$ and $H^{\iota}_{k,y_0,\epsilon}(y)$ for $i,j\in\{0,1,2\},\iota\in\{\pm\}, \sigma,\tau\in\{0,1\}$, $(k,y_0,\epsilon)\in\Sigma$, satisfying
\begin{equation}\label{P1.00}
\left\|\Upsilon^{\iota ij}_{\sigma,k,y_0,\epsilon}\right\|_{X^{1,2}_{k,y_0,\iota\epsilon}}\lesssim|k|^{-2} \big[\log{\langle k\rangle}\big]^7,
\end{equation}
such that
\begin{equation}\label{P1.0}
\begin{split}
F^{4\iota}_{k,y_0,\epsilon}(y)=&-\frac{b''(y)\psi^{\iota}_{k,\epsilon}(y,y_0)}{2|b'(y)|^2(b(y)-b(y_0)+i\iota\epsilon)}\\
&\\
                                          &+\sum_{\sigma,\tau\in\{0,1\}, i,j\in\{0,1,2\}}\Upsilon^{\iota ij}_{\sigma,k,y_0,\epsilon}(y)\Big[\log{(b(\sigma)-b(y_0)+\iota\epsilon)}\Big]^i\Big[\log{(b(\tau)-b(y_0)+\iota\epsilon)}\Big]^j.
\end{split}
\end{equation}
\end{lemma}

\begin{remark}\label{RM1}
In general, $\Upsilon^{\iota ij}_{\sigma,k,y_0,\epsilon}(y)$ are not vanishing even if we assume that $\omega_0^k(y)$ vanishes at $y=0,1$, see \eqref{P7}.
\end{remark}

We also need the following estimates.
\begin{lemma}\label{P18}
Set
\begin{equation}\label{P19}
F^{5\iota}_{k,y_0,\epsilon}:=\int_0^1G_k(y,z)b''(z)\frac{\partial_{y_0}\psi^{\iota}_{k,\epsilon}(z,y_0)}{(b(z)-b(y_0)+i\iota\epsilon)^2}dz.
\end{equation}
Then there exist functions $\Upsilon^{\iota i j}_{\sigma,\tau,k,y_0,\epsilon}(y)$ and $H^{\iota}_{k,y_0,\epsilon}(y)$ for $\sigma,\tau\in\{0,1\}, \iota\in\{\pm\},i,j\in\{0,1,2\}$, $(k,y_0,\epsilon)\in\Sigma$, satisfying
\begin{equation}\label{P19.0}
\left\|\Upsilon^{\iota i j}_{\sigma,\tau,k,y_0,\epsilon}\right\|_{X^{1,2}_{k,y_0,\iota\epsilon}}\lesssim  |k|^{-2}\big[\log{\langle k\rangle}\big]^{10},
\end{equation}
such that
\begin{equation}\label{P19.1}
F^{5\iota}_{k,y_0,\epsilon}(y)=\sum_{\sigma,\tau\in\{0,1\},i,j\in\{0,1,2\}}\Upsilon^{\iota ij}_{\sigma,\tau,k,y_0,\epsilon}(y)\Big[\log{(b(\sigma)-b(y_0)+i\iota\epsilon)}\Big]^i\Big[\log{(b(\tau)-b(y_0)+i\iota\epsilon)}\Big]^j.
\end{equation}
\end{lemma}

The proofs of Lemma \ref{P0} and Lemma \ref{P18} use only properties on $\psi^{\iota}_{k,\epsilon}(y,y_0), \partial_{y_0}\psi^{\iota}_{k,\epsilon}(y,y_0)$ obtained in Lemma \ref{L0.1} and Lemma \ref{L0.2}. The calculations are relatively straightforward but lengthy. We postpone the proofs to the appendix. 

Finally we are ready to prove the following bounds on $\partial_{y_0}^2\psi^{\iota}_{k,\epsilon}(y,y_0)$.
\begin{lemma}\label{L1}
We have  the following decomposition
\begin{equation}\label{L2.2}
\begin{split}
\partial_{y_0}^2&\psi^{\iota}_{k,\epsilon}(y,y_0)=\omega_0^k(1)\frac{|b'(y_0)|^2\Phi^{1\iota}_{k,\epsilon}(y,y_0)}{b(1)-b(y_0)+i\iota\epsilon}+\omega_0^k(0)\frac{|b'(y_0)|^2\Phi^{0\iota}_{k,\epsilon}(y,y_0)}{b(0)-b(y_0)+i\iota\epsilon}\\
&\hspace{0.7in}+\frac{|b'(y_0)|^2}{|b'(y)|^2}\times\frac{b''(y)\psi^{\iota}_{k,\epsilon}(y,y_0)-\omega_0^k(y)}{b(y)-b(y_0)+i\iota\epsilon}\\
&+\sum_{\sigma,\tau\in\{0,1\},i,j\in\{0,1,2\}}\Upsilon^{\iota ij}_{\sigma,\tau,k,y_0,\epsilon}(y)\big[\log{(b(\sigma)-b(y_0)+i\iota\epsilon)}\big]^i\big[\log{(b(\tau)-b(y_0)+i\iota\epsilon)}\big]^j,
\end{split}\end{equation}
where  $\iota\in\{\pm\}, (k,y_0,\epsilon)\in\Sigma$ and $(y,y_0)\in[0,1]$ and
\begin{equation}\label{L2.001}
\begin{split}
\sum_{\sigma\in\{0,1\}}|k|^{-1}\left\|\Phi^{\sigma\iota}_{k,\epsilon}\right\|_{Z^1_{k,y_0,\iota\epsilon}}+\sum_{\sigma,\tau\in\{0,1\},\,\iota\in\{+,-\},i,j\in\{0,1,2\}}\left\|\Upsilon^{j\iota i}_{\sigma,\tau.k,y_0,\epsilon}\right\|_{X^{1,2}_{k,y_0,\epsilon}}\lesssim |k|^{-1}\big[\log{\langle k\rangle}\big]^{10}.
\end{split}
\end{equation}
In addition, $\Phi^{\sigma\iota}_{k,\epsilon}(y,y_0)$ are given by the explicit equations \eqref{L8.31'}, and
\begin{equation}\label{L2.002}
\lim_{\epsilon\to0}\Big[\Phi^{\sigma+}_{k,\epsilon}(y,y_0)-\Phi^{\sigma-}_{k,\epsilon}(y,y_0)\Big]=0,\qquad{\rm for}\,\,y_0\in\{0,1\}, \sigma\in\{0,1\},
\end{equation}
\begin{equation}\label{L2.003}
\sup_{\sigma\in\{0,1\}}\left|\partial_{y_0,\epsilon}\Phi^{\sigma\iota}_{k,\epsilon}(y,y_0)\right|\lesssim\big[\log{\langle k\rangle}\big]^{10}\bigg[\sum_{\alpha\in\{0,1,y\}}\big|\log{(b(\alpha)-b(y_0)+i\iota\epsilon)}\big|\bigg].
\end{equation}
\end{lemma}

\begin{proof}
Taking two derivatives in $y_0$ in \eqref{F8} we obtain
\begin{equation}\label{L8.1}
\begin{split}
\partial^2_{y_0}&\psi^{\iota}_{k,\epsilon}(y,y_0)+\int_0^1G_k(y,z)\frac{b''(z)\partial^2_{y_0}\psi^{\iota}_{k,\epsilon}(z,y_0)}{b(z)-b(y_0)+i\iota\epsilon}\,dz\\
=&b''(y_0)F^{2\iota}_{k,y_0,\epsilon}+2|b'(y_0)|^2F^{3\iota}_{k,y_0,\epsilon}-2(b'(y_0))^2F^{4\iota}_{k,y_0,\epsilon}(y)-2b'(y_0)F^{5\iota}_{k,y_0,\epsilon}(y)\\
&-b''(y_0)\int_0^1G_k(y,z)\frac{b''(z)\psi^{\iota}_{k,\epsilon}(z,y_0)}{(b(z)-b(y_0)+i\iota\epsilon)^2}dz:=\mathcal{N}.
\end{split}
\end{equation}
By Lemma \ref{L0.3}-Lemma \ref{P18}, we can represent 
$\mathcal{N}=\sum_{j=1}^3\mathcal{N}_j,$ where
\begin{equation}\label{L8.2}
\begin{split}
\mathcal{N}_1:=&\frac{|b'(y_0)|^2\sinh{(ky)}}{|b'(1)|^2\sinh{k}}\frac{\omega_0^k(1)}{b(1)-b(y_0)+i\iota\epsilon}+\frac{|b'(y_0)|^2\sinh{(k(1-y))}}{|b'(0)|^2\sinh{k}}\frac{\omega_0^k(0)}{b(0)-b(y_0)+i\iota\epsilon},\\
\mathcal{N}_2:=&\frac{|b'(y_0)|^2}{|b'(y)|^2}\times\frac{b''(y)\psi^{\iota}_{k,\epsilon}(y,y_0)-\omega_0^k(y)}{b(y)-b(y_0)+i\iota\epsilon},\\
 \mathcal{N}_3:=&\sum_{\sigma,\tau\in\{0,1\},i,j\in\{0,1,2\}}\Upsilon^{\ast\iota ij}_{\sigma,\tau,k,y_0,\epsilon}(y)\big[\log{(b(\sigma)-b(y_0)+i\iota\epsilon)}\big]^i\big[\log{(b(\tau)-b(y_0)+i\iota\epsilon)}\big]^j,
 \end{split}
\end{equation}
for some $\Upsilon^{\ast\iota i j}_{\sigma,\tau,k,y_0,\epsilon}(y)$ satisfying
\begin{equation}\label{L8.3}
\sum_{i,j\in\{0,1,2\},\,\sigma,\tau\in\{0,1\}}\left\|\Upsilon^{\ast\iota ij}_{\sigma,\tau,k,y_0,\epsilon}\right\|_{X^{1,2}_{k,y_0,\iota\epsilon}}\lesssim |k|^{-1}\big[\log{\langle k\rangle}\big]^{10}.
\end{equation}
We solve the equation
\begin{equation}\label{L8.31}
f(y)+\int_0^1G_k(y,z)\frac{f(z)}{b(z)-b(y_0)+i\iota\epsilon}=\mathcal{N}_i,
\end{equation}
for $i\in\{1,2,3\}$ respectively, and then $\partial_{y_0}^2\psi^{\iota}_{k,\epsilon}$ is the sum of the corresponding solutions. The case of $\mathcal{N}_1$ follows from Lemma \ref{X22} with the norm $Y^1_{k,y_0,\iota\epsilon}$. We note that the functions $\Phi^{\sigma\iota}_{k,\epsilon}, \sigma\in\{0,1\}$ solves
\begin{equation}\label{L8.31'}
\begin{split}
(I-S^{\iota}_{k,y_0,\epsilon})\Phi^{1\iota}_{k,\epsilon}&=\frac{\sinh{(ky)}}{|b'(1)|^2\sinh{k}},\\
(I-S^{\iota}_{k,y_0,\epsilon})\Phi^{0\iota}_{k,\epsilon}&=\frac{\sinh{(k(1-y))}}{|b'(0)|^2\sinh{k}}.
\end{split}
\end{equation}
The claims \eqref{L2.001}-\eqref{L2.003} on $\Phi^{j\iota}_{k,\epsilon}, j\in\{2,3\}$ follow from \eqref{L8.31'}, in view of Lemma \ref{X22}, similar to the proof of \eqref{L0.20001}.
The case of $\mathcal{N}_3$ follow from Lemma \ref{X22} with the norm $X^{1,2}_{k,y_0,\iota\epsilon}$.

The only nontrivial case is $\mathcal{N}_2$, which is much more singular than Lemma \ref{X22} would allow. In this case we write the solution to \eqref{L8.31} with $i=2$ in the form of
\begin{equation}\label{L8.32}
f^{\iota}_{k,y_0,\epsilon}(y)+\frac{|b'(y_0)|^2}{|b'(y)|^2}\times\frac{b''(y)\psi^{\iota}_{k,\epsilon}(y,y_0)-\omega_0^k(y)}{b(y)-b(y_0)+i\iota\epsilon},
\end{equation}
Then $f^{\iota}_{k,y_0,\epsilon}(y)$ solves
\begin{equation}\label{L8.33}
f^{\iota}_{k,y_0,\epsilon}(y)-S^{\iota}_{k,y_0,\epsilon}f^{\iota}_{k,y_0,\epsilon}(y)=S^{\iota}_{k,y_0,\epsilon}\left[\frac{|b'(y_0)|^2}{|b'(y)|^2}\times\frac{b''(y)\psi^{\iota}_{k,\epsilon}(y,y_0)-\omega_0^k(y)}{b(y)-b(y_0)+i\iota\epsilon}\right]:=\mathcal{N}_2'.
\end{equation}
Now by Lemma \ref{bX17} and the definitions \eqref{Sep} and \eqref{L2.1}, we can bound  $f^{\iota}_{k,y_0,\epsilon}(y)$ using Lemma \ref{bX17} in the norm $Y^{1}_{k,y_0,\epsilon}$ for the boundary term of $\mathcal{N}_2'$ and the norm $X^{1,2}_{k,y_0,\epsilon}$ for the other terms of $\mathcal{N}_2'$. The Lemma then follows.  

\end{proof}

\section{Proof of the main theorem}
With the help of Lemma \ref{L0.1}, Lemma \ref{L0.2} and Lemma \ref{L1}, we can now prove the precise decay rate of the stream function $\psi_k(t,y)$ and the main theorem \ref{thm}. 
\begin{lemma}\label{L10}
The stream function $\psi_k(t,y)$ satisfies for each $k\in\mathbb{Z}\backslash\{0\}$,
\begin{equation}\label{L11}
\begin{split}
\psi_k(t,y)=&\frac{e^{-ikb(y)t}}{k^2t^2}\frac{b''(y)}{|b'(y)|^2}\bigg[\varphi_{1k}(y)-\varphi_{2k}(y)\mathbf{1}_{k<0}\bigg]-\frac{e^{-ikb(y)t}}{k^2t^2}\frac{\omega_0^k(y)}{|b'(y)|^2}\\
                  &+\frac{e^{-ikb(0)t}}{2k^2\pi t^2i|b'(0)|}\varphi_{3k}(y)-\frac{e^{-ikb(1)t}}{2k^2\pi t^2i|b'(1)|}\varphi_{4k}(y)\\
                  &+ \omega_0^k(0)\frac{e^{-ikb(0)t}}{k^2t^2}\varphi_{5k}(y)+\omega_0^k(1)\frac{e^{-ikb(1)t}}{k^2t^2}\varphi_{6k}(y)+\Gamma_{1k}(t,y),
\end{split}
\end{equation}
and
\begin{equation}\label{L11.1}
\partial_y\psi_k(t,y)= i\frac{e^{-ikb(y)t}}{kt}\frac{\omega_0^k(y)}{b'(y)}- i\frac{e^{-ikb(y)t}}{kt}\frac{b''(y)}{b'(y)}\Big[\varphi_{1k}(y)-\varphi_{2k}(y)\mathbf{1}_{k<0}\Big]+\Gamma_{2k}(t,y).
\end{equation}
In the above we recall the definitions \eqref{F6}--\eqref{F7} and \eqref{L8.31'}, and set 
\begin{equation}\label{L11.0}
\varphi_{1k}(y):=\lim_{\epsilon\to0+}\psi^{+}_{k,\epsilon}(y,y),\qquad \varphi_{2k}(y):=\lim_{\epsilon\to0+}\left[\psi^{-}_{k,\epsilon}(y,y)-\psi^{+}_{k,\epsilon}(y,y)\right].
\end{equation}
\begin{equation}\label{L11.00}
\varphi_{3k}(y):=\lim_{\epsilon\to0+}\big[\partial_{y_0}\psi^{-}_{k,\epsilon}(y,0)-\partial_{y_0}\psi^{+}_{k,\epsilon}(y,0)\big], \qquad\varphi_{4k}(y):=\lim_{\epsilon\to0+}\big[\partial_{y_0}\psi^{-}_{k,\epsilon}(y,1)-\partial_{y_0}\psi^{+}_{k,\epsilon}(y,1)\big],
\end{equation}
\begin{equation}\label{L11.01}
\varphi_{5k}(y):=\lim_{\epsilon\to0+}\Phi^{0+}_{k,\epsilon}(y,0),\qquad \varphi_{6k}(y):=\lim_{\epsilon\to0+}\Phi^{1+}_{k,\epsilon}(y,1);
\end{equation}
In addition, with the normalization $\|\omega_0^k\|_{H^3_k}=1$, (recall the definition \eqref{Hk}) we have
\begin{equation}\label{L11.02}
\Gamma_{1 k}(t,y)=\frac{1}{2\pi ik^2t^2}\int_0^1e^{-ikb(y_0)t}\Upsilon_{1 k}(y,y_0)\,dy_0+\Gamma_{1 k}'(t,y),
\end{equation} 
\begin{equation}\label{L11.2}
\qquad\|\Gamma_{1 k}'(t,y)\|_{L^{\infty}}\lesssim \frac{|k|^{-1}}{(kt)^{23/8}}, \qquad \|\Gamma_{2 k}(t,y)\|_{L^{\infty}}\lesssim \frac{1}{|k|^{1.9}t^{15/8}};
\end{equation}
and for $\sigma\in\{1,2\}$,
\begin{equation}\label{L11.3}
|\Upsilon_{\sigma k}(y,y_0)|\lesssim |k|^{-1}\big[\log{\langle k\rangle}\big]^{10}\bigg[\sum_{\alpha\in\{0,1,y\}}\big|\log{|y_0-\alpha|}\big|+1\bigg]^6,
\end{equation}
Moreover, if $\omega_0^k(0)=\omega_0^k(1)=0$, then
\begin{equation}\label{L11.4}
\varphi_{3k}(y)\equiv\varphi_{4k}(y)\equiv0,
\end{equation}
and in this case the decomposition for $\psi_k(t,y)$ simplies
\begin{equation}\label{L11.5}
\begin{split}
\psi_k(t,y)=\frac{e^{-ikb(y)t}}{k^2t^2}\frac{b''(y)}{|b'(y)|^2}\Big[\varphi_{1k}(y)-\varphi_{2k}(y)\mathbf{1}_{k<0}\Big]-\frac{e^{-ikb(y)t}}{k^2t^2}\frac{\omega_0^k(y)}{|b'(y)|^2}+\Gamma_{1k}(t,y).
\end{split}
\end{equation}
\end{lemma}

\begin{remark}\label{MR9}
It is possible to obtain more quantitative decay estimates on $\Gamma_{1k}(t,y)$ than the qualitative bounds in \eqref{L11.02}. However, it would  require a precise understanding of the singularities of $\partial_{y_0}^3\psi^{\iota}_{k,\epsilon}(y,y_0)$, $\iota\in\{\pm\}$. While it can be done using ideas in this paper, the computations involved are lengthy, especially in the presence of boundary terms, which always need to be tracked separately.

The precise asymptotic \eqref{L11.1} and \eqref{L11.5} could be useful for nonlinear applications. In fact, based on the main terms in \eqref{L11.5}, it is tempting to speculate that for the nonlinear problem the correct quantity to track is not $\omega$ but a suitable modification of $\omega$, adapted to the asymptotic given by \eqref{L11.1}. 
\end{remark}

\begin{proof}
We normalize $\|\omega_0^k\|_{H^3}=1$.\\
{\bf Step 1: The proof of \eqref{L11}.}
Our starting point is the formula \eqref{F5}, which we reproduce here
\begin{equation}\label{L14}
 \psi_k(t,y)=-\frac{1}{2\pi i}\lim_{\epsilon\to0+}\int_{0}^1e^{-ikb(y_0) t}|b'(y_0)|\left[\psi_{k,\epsilon}^{-}(y_0,y)-\psi_{k,\epsilon}^{+}(y,y_0)\right]dy_0.
\end{equation}
The basic idea is to use integration by parts in $y_0$ in the formula \eqref{L14} to gain decay in $t$. Let $\aleph$ be the sign of $b'$. We first note, using \eqref{L0.20002}-\eqref{L0.20003}, that
\begin{equation}\label{L14.0}
\left|\psi^{+}_{k,\epsilon}(y,y_0)-\psi^{-}_{k,\epsilon}(y,y_0)\right|\lesssim |k|^{-2}\big[\log{\langle k\rangle}\big]^4\epsilon^{1/2},\qquad{\rm for}\,\,y_0\in\{0,1\}.
\end{equation}
By Lemma \ref{L0.2}, we then obtain from \eqref{L14.0} that
\begin{equation}\label{L14.1}
\left|\lim_{\epsilon\to0}\big[\psi^{+}_{k,y_0,\epsilon}(y)-\psi^{-}_{k,y_0,\epsilon}(y)\big]\right|\lesssim|k|^{-2}\big[\log{\langle k\rangle}\big]^4 \min\left\{|y_0-1|^{1/2},|y_0|^{1/2}\right\}.
\end{equation}
Integration by parts in $y_0$, we obtain
\begin{equation}\label{L15}
\begin{split}
 \aleph\psi_k(t,y)=&\frac{1}{2k\pi t}\lim_{\epsilon\to0+}\int_{0}^1e^{-ikb(y_0) t}\left[\partial_{y_0}\psi_{k,\epsilon}^{-}(y,y_0)-\partial_{y_0}\psi_{k,\epsilon}^{+}(y,y_0)\right]dy_0\\
                 =&-\frac{1}{2k^2\pi it^2}\lim_{\epsilon\to0+}\int_{0}^1e^{-ikb(y_0) t}\frac{b''(y_0)}{|b'(y_0)|^2}\left[\partial_{y_0}\psi_{k,\epsilon}^{-}(y,y_0)-\partial_{y_0}\psi_{k,\epsilon}^{+}(y,y_0)\right]dy_0\\
                   &-\frac{1}{2k^2\pi i t^2}\frac{e^{-ikb(y_0) t}}{b'(y_0)}\lim_{\epsilon\to0+}\left[\partial_{y_0}\psi_{k,\epsilon}^{-}(y,y_0)-\partial_{y_0}\psi_{k,\epsilon}^{+}(y,y_0)\right]\bigg|_{y_0=0}^{1}\\
                   &+\frac{1}{2k^2\pi t^2i}\lim_{\epsilon\to0+}\int_0^1\frac{e^{-ikb(y_0) t}}{b'(y_0)}\left[\partial_{y_0}^2\psi_{k,\epsilon}^{-}(y,y_0)-\partial_{y_0}^2\psi_{k,\epsilon}^{+}(y,y_0)\right]dy_0=\mathcal{T}_3+\mathcal{T}_4+\mathcal{T}_5.
 \end{split}
\end{equation}

{\bf Substep 1.1} From the decomposition \eqref{L0.21} and \eqref{L2.2}, it follows from integration by parts argument that
\begin{equation}\label{L15.1}
\big|\mathcal{T}_3\big|\lesssim  \frac{|k|^{-1}}{(kt)^{23/8}}.
\end{equation}

{\bf Substep 1.2} 
We now consider the term $\mathcal{T}_4$. Set for $\sigma\in\{0,1\}$
\begin{equation}\label{L16}
\beta_{\sigma k}(y):=\lim_{\epsilon\to0+}\Big[\partial_{y_0}\psi_{k,\epsilon}^{-}(y,\sigma)-\partial_{y_0}\psi_{k,\epsilon}^{+}(y,\sigma)\Big].
\end{equation}
By Lemma \ref{L0.2}, 
\begin{equation}\label{L16.1}
\sum_{\sigma\in\{0,1\}}\left\|\beta_{\sigma k}\right\|_{L^{\infty}_y}\lesssim|k|^{-2} \big[\log{\langle k\rangle}\big]^7.
\end{equation}
By the definitions, we conclude that 
\begin{equation}\label{L16.2}
\mathcal{T}_4=\frac{1}{2k^2\pi i t^2}\frac{e^{-iktb(0)}}{b'(0)}\beta_{0k}(y)-\frac{1}{2k^2\pi i t^2}\frac{e^{-iktb(1)}}{b'(1)}\beta_{1k}(y).
\end{equation}
We also note, if $\omega_0^k(0)=\omega_0^k(1)=0$, then the functions $\beta_{\sigma k}(y)\equiv0$ for $\sigma\in\{0,1\}, y\in[0,1]$, which can be seen from \eqref{L6}.

{\bf Substep 1.3}
We now consider the term $\mathcal{T}_5$. We use \eqref{L2.2}, and obtain from \eqref{L15} that
\begin{equation}\label{L16.4}
\begin{split}
\mathcal{T}_5=&\frac{1}{2k^2\pi t^2i}\frac{b''(y)}{|b'(y)|^2}\lim_{\epsilon\to0}\int_0^1e^{-ikb(y_0)t}b'(y_0)\left[\frac{\psi_{k,\epsilon}^{-}(y,y_0)}{b(y)-b(y_0)-i\epsilon}-\frac{\psi^{+}_{k,\epsilon}(y,y_0)}{b(y)-b(y_0)+i\epsilon}\right]dy_0\\
                         &+\sum_{\sigma\in\{0,1\}}\omega_0^k(\sigma)\frac{1}{2k^2\pi t^2i}\int_0^1e^{-ikb(y_0)t}b'(y_0)\left[\frac{\Phi^{\sigma-}_{k,\epsilon}(y,y_0)}{b(\sigma)-b(y_0)-i\epsilon}-\frac{\Phi^{\sigma+}_{k,\epsilon}(y,y_0)}{b(\sigma)-b(y_0)+i\epsilon}\right]\,dy_0\\
                       &-\aleph\frac{\omega_0^k(y)}{|b'(y)|^2}\frac{e^{-ikb(y)t}}{k^2t^2}+\frac{1}{2\pi ik^2t^2}\int_0^1e^{-ikb(y_0)t}b'(y_0)\Upsilon_{k,\epsilon}(y,y_0)dy_0\\
:=&\frac{b''(y)}{|b'(y)|^2}\mathcal{T}_{51}+\mathcal{T}_{52}+\frac{1}{2\pi ik^2t^2}\int_0^1e^{-ikb(y_0)t}b'(y_0)\Upsilon_{k,\epsilon}(y,y_0)dy_0-\aleph\frac{\omega_0^k(y)}{|b'(y)|^2}\frac{e^{-ikb(y)t}}{k^2t^2}.
\end{split}
\end{equation}
In the above, we used the definitions \eqref{L8.31'}, and $\Upsilon_{k,\epsilon}(y,y_0)$ are given as in \eqref{L11.3}.

We first consider the term
\begin{equation}\label{L17.4}
\begin{split}
\mathcal{T}_{51}:&=\frac{1}{2k^2\pi i t^2}\lim_{\epsilon\to0+}\int_0^1e^{-ikb(y_0) t}b'(y_0)\left[\frac{\psi^{-}_{k,\epsilon}(y,y_0)}{b(y)-b(y_0)-i\epsilon}-\frac{\psi^{+}_{k,\epsilon}(y,y_0)}{b(y)-b(y_0)+i\epsilon}\right]\,dy_0\\
&=\frac{1}{2k^2\pi i t^2}\lim_{\epsilon\to0+}\int_{0}^1e^{-ikb(y_0) t}b'(y_0)\left[\frac{\psi^{-}_{k,\epsilon}(y,y_0)-\psi^{+}_{k,\epsilon}(y,y_0)}{b(y)-b(y_0)-i\epsilon}\right]\,dy_0\\
&\qquad+\frac{1}{k^2\pi  t^2}\lim_{\epsilon\to0+}\epsilon\int_{0}^1e^{-ikb(y_0) t}b'(y_0)\left[\frac{ \psi^{+}_{k,\epsilon}(y,y_0)}{(b(y)-b(y_0))^2+\epsilon^2}\right]\,dy_0\\
&=\frac{1}{2k^2\pi t^2i}\lim_{\epsilon\to0+}\int_{0}^1e^{-ikb(y_0) t}b'(y_0)\left[\frac{\psi^{-}_{k,\epsilon}(y,y_0)-\psi^{+}_{k,\epsilon}(y,y_0)}{b(y)-b(y_0)-i\epsilon}\right]\,dy_0+\aleph\frac{e^{-ikb(y)t}}{k^2t^2}\beta_{2k}(y),
\end{split}
\end{equation}
where
\begin{equation}\label{L17.5}
\beta_{2k}(y):=\lim_{\epsilon\to0+}\psi^{+}_{k,\epsilon}(y,y)\qquad {\rm and}\qquad\|\beta_{2k}\|_{L^{\infty}}\lesssim |k|^{-2}.
\end{equation}
Set 
\begin{equation}\label{L21.2}
\beta_{3k}(y,y_0):=\lim_{\epsilon\to0+}\left[\psi^{-}_{k,\epsilon}(y,y_0)-\psi^{+}_{k,\epsilon}(y,y_0)\right]\mathbf{1}_{y_0\in[0,1]}.
\end{equation}
Write
\begin{equation}\label{L21.4}
\mathcal{T}_{51}=\mathcal{T}_{51}'+\aleph\frac{e^{-ikb(y)t}}{k^2t^2}\beta_{2k}(y).
\end{equation}
It remains to bound the term 
\begin{equation}\label{L22}
\begin{split}
\mathcal{T}_{51}':&=\frac{1}{2k^2\pi t^2i}\lim_{\epsilon\to0+}\int_{0}^1e^{-ikb(y_0) t}b'(y_0)\left[\frac{\psi^{-}_{k,\epsilon}(y,y_0)-\psi^{+}_{k,\epsilon}(y,y_0)}{b(y)-b(y_0)-i\epsilon}\right]\,dy_0\\
                           &=\frac{1}{2k^2\pi t^2i}\lim_{\epsilon\to0+}\int_{0}^1e^{-ikb(y_0) t}b'(y_0)\frac{\beta_{3k}(y,y_0)}{b(y)-b(y_0)-i\epsilon}\,dy_0\\
                            &=\frac{\aleph}{2k^2\pi t^2i}\lim_{\epsilon\to0+}\int_{\R}e^{-ikz t}\frac{\beta_{3k}(y,b^{-1}(z))}{b(y)-z+i\epsilon}\,dz.
\end{split}
\end{equation}
In the above we used \eqref{L0.20003}, and assume an appropriate monotone extension of $b$ to $\R$.

In view of Lemma \ref{L0.1}, Lemma \ref{L0.2} and Lemma \ref{L1}, using Fourier transform, we can thus find $f_{k}(y,\xi)$ with
\begin{equation}\label{L24}
\sup_{y\in[0,1],\xi\in\R}\left|\langle\xi\rangle^{31/16} f_k(y,\xi)\right|\lesssim|k|^{-1}\big[\log{\langle k\rangle}\big]^{10},
\end{equation}
such that
\begin{equation}\label{L25}
\beta_{3k}(y,b^{-1}(z))=\int_{\R} f_k(y,\xi)e^{iz\xi}\,d\xi.
\end{equation}
Using \eqref{L25}, we get
\begin{equation}\label{L26}
\begin{split}
\mathcal{T}_{51}'&=\frac{\aleph}{2k^2\pi t^2i}\lim_{\epsilon\to0+}\int_{\R}\int_{\R}e^{-ikz t+iz\xi}\frac{f_k(y,\xi)}{b(y)-z+i\epsilon}\,dzd\xi\\
                           &=-\aleph\frac{e^{-ikb(y)t}}{k^2t^2}\int_{kt<\xi}f_k(y,\xi)e^{ib(y)\xi}\,d\xi.
\end{split}
\end{equation}
Thus, in view of \eqref{L24}, it follows that
\begin{equation}\label{L27}
\left\|\mathcal{T}_{51}'+\aleph\frac{e^{-ikb(y)t}}{k^2t^2}\mathbf{1}_{k<0}\beta_{3k}(y,y)\right\|_{L^{\infty}}\lesssim \frac{|k|^{-1}}{(kt)^{23/8}}.
\end{equation}

Thus from \eqref{L21.4} and \eqref{L27}, we conclude that
\begin{equation}\label{L27.1}
\mathcal{T}_{51}=\aleph\frac{e^{-ikb(y)t}}{k^2t^2}\beta_{2k}(y)-\aleph\frac{e^{-ikb(y)t}}{k^2t^2}\mathbf{1}_{k<0}\beta_{3k}(y,y)+\Gamma_{2k}(t,y),
\end{equation}
where
\begin{equation}\label{L27.2}
\|\Gamma_{2k}(t)\|_{L^{\infty}}\lesssim \frac{|k|^{-1}}{(kt)^{23/8}}.
\end{equation}

Completely analogous to the treatment of the term $\mathcal{T}_{51}$, set 
\begin{equation}\label{L27.3}
\beta_{5k}(y,y_0):=\lim_{\epsilon\to0+}\left[\Phi^{1-}_{k,\epsilon}(y,y_0)-\Phi^{1+}_{k,\epsilon}(y,y_0)\right], \qquad \beta_{6k}(y,y_0):=\lim_{\epsilon\to0+}\left[\Phi^{0-}_{k,\epsilon}(y,y_0)-\Phi^{0+}_{k,\epsilon}(y,y_0)\right],
\end{equation}
\begin{equation}\label{L27.4}
\beta_{7k}(y):=\lim_{\epsilon\to0+}\Phi^{1+}_{k,\epsilon}(y,1),\qquad \beta_{8k}(y):=\lim_{\epsilon\to0+}\Phi^{0+}_{k,\epsilon}(y,0),
\end{equation}
then
\begin{equation}\label{L27.5}
\begin{split}
\mathcal{T}_{52}:=&\aleph\omega_0^k(1)\frac{e^{-ikb(1)t}}{k^2t^2}\Big[\beta_{7k}(y)-\mathbf{1}_{k<0}\beta_{5k}(y,1)\Big]\\
                              &+\aleph\omega_0^k(0)\frac{e^{-ikb(0)t}}{k^2t^2}\Big[\beta_{8k}(y)-\mathbf{1}_{k<0}\beta_{6k}(y,0)\Big]+\Gamma_{3k}(t,y)\\
                            =& \aleph \omega_0^k(1)\frac{e^{-ikb(1)t}}{k^2t^2}\beta_{7k}(y)+\aleph\omega_0^k(0)\frac{e^{-ikb(0)t}}{k^2t^2}\beta_{8k}(y)+\Gamma_{3k}(t,y),
\end{split}
\end{equation}
where
\begin{equation}\label{L27.6}
\|\Gamma_{3k}(t)\|_{L^{\infty}}\lesssim \frac{|k|^{-1}}{(kt)^{23/8}}.
\end{equation}

Combining the bounds on the terms $\mathcal{T}_{51}$ and $\mathcal{T}_{52}$, using \eqref{L16.4} we get bounds on $\mathcal{T}_5$, which together with the expressions \eqref{L15.1} on $\mathcal{T}_3$ and \eqref{L16.2} on $\mathcal{T}_4$, completes the proof of \eqref{L11}.\\
\noindent
{\bf Step 2: The proof of \eqref{L11.1}.} The proof of \eqref{L11.1} follows similar line, using the formula (see the first line of \eqref{L15})
\begin{equation}\label{L28}
\aleph\partial_y\psi_k(t,y)=\frac{1}{2k\pi t}\lim_{\epsilon\to0+}\int_0^1e^{-ikb(y_0)t}\Big[\partial_y\partial_{y_0}\psi^{-}_{k,\epsilon}(y,y_0)-\partial_y\partial_{y_0}\psi^{+}_{k,\epsilon}(y,y_0)\Big]\,dy_0
\end{equation} 
and the decomposition \eqref{L0.22} in Lemma \eqref{L0.2}. We reformulate \eqref{L0.22} as
\begin{equation}\label{L29}
\partial_{y}\partial_{y_0}\psi^{\iota}_{k,\epsilon}(y,y_0)=\left[\frac{b'(y_0)\omega^k_0(y)}{b'(y)}-\frac{b'(y_0)b''(y)}{b'(y)}\psi^{\iota}_{k,\epsilon}(y,y_0)\right]\frac{1}{b(y)-b(y_0)+i\iota\epsilon}+\mathcal{R}^{\iota}_{k,\epsilon}(y,y_0).
\end{equation}
Comparing the definition \eqref{L17.4} of $\mathcal{T}_{51}$, with a completely analogous argument as in the treatment of $\mathcal{T}_{51}$, we obtain from \eqref{L28} that
\begin{equation}\label{L29.1}
\begin{split}
\aleph\partial_{y}\psi_k(t,y)=&i\frac{e^{-ikb(y)t}}{kt}\frac{\omega_0^k(y)}{|b'(y)|}-kit\frac{b''(y)}{b'(y)}\mathcal{T}_{51}\\
                                                                                       &+\frac{1}{2k\pi t}\int_0^1e^{-ikb(y_0)t}\left[\mathcal{R}^{-}_{k,\epsilon}(y,y_0)-\mathcal{R}^{+}_{k,\epsilon}(y,y_0)\right]\,dy_0.
\end{split}
\end{equation}
Using Lemma \ref{L0.2}, comparing \eqref{L0.22} with \eqref{L2.2}-\eqref{L2.001}, we obtain the following bounds for $\mathcal{R}^{\iota}_{k,\epsilon}(y,y_0)$:
\begin{equation}\label{L30}
\begin{split}
\left|\mathcal{R}^{\iota}_{k,\epsilon}(y,y_0)\right|\lesssim|k|^{-1}\Big[\log{\langle k\rangle}\Big]^{10}\Big[\sum_{\alpha\in\{0,1,y\}}\big|\log{|y_0-\alpha|}\big|+1\Big]^6,
\end{split}
\end{equation}
\begin{equation}\label{L31}
\begin{split}
\left|\partial_{y_0}\mathcal{R}^{\iota}_{k,\epsilon}(y,y_0)\right|\lesssim \Big[\log{\langle k\rangle}\Big]^{10}\Big[\sum_{\alpha\in\{0,1,y\}}\big|\log{|y_0-\alpha|}\big|+1\Big]^6  \bigg[\sum_{\alpha\in\{0,1,y\}}\frac{1}{|\alpha-y_0|}\bigg].
\end{split}
\end{equation}
Using \eqref{L30}-\eqref{L31}, it is clear that for $\iota\in\{\pm\}$,
\begin{equation}\label{L32}
\left|\frac{1}{2k\pi t}\int_0^1e^{ikb(y_0)t}\mathcal{R}^{\iota}_{k,\epsilon}(y,y_0)\,dy_0\right|\lesssim \frac{1}{|k|^{1.9}t^{1.9}}.
\end{equation}
In view of \eqref{L29.1}, \eqref{L27.1} and \eqref{L32}, we completed the proof of \eqref{L11.1}.
\end{proof}

We finally give the proof of the main theorem.
\subsection{Proof of Theorem \ref{thm}}
In view of the definitions \eqref{Th1}, we have
\begin{equation}\label{L40}
\partial_tf(t,x,y)=b''(y)\partial_x\phi(t,x,y)=C_0b''(y)\sum_{k\in\mathbb{Z}}ike^{ikb(y)t+ikx}\psi_k(t,y).
\end{equation}
The bounds and convergence results \eqref{Th2} follow from the asymptotic formula \eqref{L11} which implies the right hand side of \eqref{L40} has integrable in time decay with the required bounds.

The bounds \eqref{Th3} follows from \eqref{L11} and \eqref{L11.1}, in view of the formulae
\begin{equation}\label{40.1}
\phi(t,x,y)=C_0\sum_{k\in\mathbb{Z}}e^{ikb(y)t+ikx}\psi_k(t,y),
\end{equation}
We now turn to \eqref{Th4} and focus only on the convergence of $\partial_yf$, which is harder.  We use \eqref{40.1} and compare the main terms in \eqref{L11.1} and \eqref{L11.5}, and obtain that
\begin{equation}\label{L42}
\begin{split}
\partial_y\phi(t,x,y)&=\frac{1}{2\pi}\sum_{k\in\mathbb{Z}}e^{ikx}\bigg[ikb'(y)t\,\Gamma_{1k}(t,y)+\Gamma_{2k}(t,y)\bigg]\\
                            &=\frac{b'(y)}{2k\pi t}\int_0^1e^{-ikb(y_0)t}\Upsilon_k(y,y_0)\,dy_0+ib'(y)kt\Gamma_{1k}'(t,y)+\Gamma_{2k}(t,y).
\end{split}
\end{equation}
An inspection of the proof of Lemma \ref{L10} shows that we can write 
\begin{equation}\label{L46}
\partial_x\partial_y\phi(t,x,y)=\sum_{k\in\mathbb{Z}}\bigg[\int_0^1e^{-ikb(y_0)t}F_{1k}(y,y_0)/t\,dy_0+O\left(\frac{1}{|k|^{0.9}t^{15/8}}\right)\|\omega_0^k\|_{H^3_k}\bigg],
\end{equation}
with $F_{1k}, F_{2k}$ satisfying the bounds
\begin{equation}\label{L46.1}
|F_{1k}(y,y_0)|\lesssim |k|^{-1}\big[\log{\langle k\rangle}\big]^{10}\bigg[\sum_{\alpha\in\{0,1,y\}}\big|\log{|y_0-\alpha|}\big|+1\bigg]^6\|\omega_0^k\|_{H^3_k}.
\end{equation}
Thanks to \eqref{L40}, to prove \eqref{Th4}, it sufficies to show
\begin{equation}\label{L43}
\sup_{T\in[0,\infty)}\left|\int_1^T\partial_x\partial_y\phi(t,x,y)dt\right|\lesssim \|\omega_0\|_{H^3},\qquad{\rm and}\qquad
\limsup_{T'>T\to\infty}\left|\int_T^{T'}\partial_x\partial_y\phi(t,x,y)dt\right|= 0.
\end{equation}
\eqref{L43} follows easily from \eqref{L46}-\eqref{L46.1}. 

\eqref{Th5} follows from the decompositions \eqref{L11.1} and \eqref{L11.5}, with
\begin{equation}\label{Psi6}
\Psi(x,y)=\frac{C_0}{|b'(y)|^2}\sum_{k\in\mathbb{Z}\backslash\{0\}}\frac{e^{ikx}}{k^2}\bigg\{\lim_{\epsilon\to0+}b''(y)\Big[\psi^{+}_{k,\epsilon}(y,y)-\mathbf{1}_{k<0}\big(\psi^{-}_{k,\epsilon}(y,y)-\psi^{+}_{k,\epsilon}(y,y)\big)\Big]-\omega_0^k(y)\bigg\}.
\end{equation}
In the above formula, we recall the definitions \eqref{F6}. The theorem is now proved.

\appendix
\section{Two technical lemmas}\label{appendix} 
In this section, we provide proofs for Lemma \ref{P0} and Lemma \ref{P18}.
\subsection{Proof of Lemma \ref{P0}} By calculations similar to \eqref{fF4.0}-\eqref{fF5}, we get that
 \begin{equation}\label{P3}
 \begin{split}
 F^{4\iota}_{k,y_0,\epsilon}(y)=&-\frac{1}{2}\frac{b''(y)\psi^{\iota}_{k,\epsilon}(y,y_0)}{|b'(y)|^2}\frac{1}{b(y)-b(y_0)+i\iota\epsilon}+\int_0^1\frac{\partial_zG_k(y,z)\partial_z\big[b''(z)\psi^{\iota}_{k,\epsilon}(z,y_0)\big]}{|b'(z)|^2\big[b(z)-b(y_0)+i\iota\epsilon\big]}dz\\
 &-\frac{3}{2}\int_0^1\frac{b''(z)}{(b'(z))^3}\frac{\big[b''(z)\psi^{\iota}_{k,\epsilon}(z,y_0)\big]\partial_zG_k(y,z)}{b(z)-b(y_0)+i\iota\epsilon}dz\\
 &-\frac{3}{2}\int_0^1\frac{b''(z)}{(b'(z))^3}\frac{\partial_z\big[b''(z)\psi^{\iota}_{k,\epsilon}(z,y_0)\big]G_k(y,z)}{b(z)-b(y_0)+i\iota\epsilon}dz\\
 &+\frac{1}{2}\int_0^1\frac{G_k(y,z)\partial_z^2\big[b''(z)\psi^{\iota}_{k,\epsilon}(z,y_0)\big]}{|b'(z)|^2(b(z)-b(y_0)+i\iota\epsilon)}\,dz+\frac{k^2}{2}\int_0^1\frac{b''(z)}{|b'(z)|^2}\frac{G_k(y,z)\psi^{\iota}_{k,\epsilon}(z,y_0)}{b(z)-b(y_0)+i\iota\epsilon}dz\\
 &+\frac{1}{2}\int_0^1\left[\frac{|b''(z)|^2}{|b'(z)|^4}+\frac{1}{b'(z)}\partial_z^2\left(\frac{1}{b'(z)}\right)\right]\frac{G_k(y,z)b''(z)\psi^{\iota}_{k,\epsilon}(z,y_0)}{b(z)-b(y_0)+i\iota\epsilon}dz:=\sum_{j=1}^7T_j.
 \end{split}
 \end{equation}
$T_1$ is the main term in \eqref{P1.0}. Upon integration by parts using 
\begin{equation}\label{P4'}
\partial_z\log{(b(z)-b(y_0)+i\iota\epsilon)}=\frac{b'(z)}{b(z)-b(y_0)+i\iota\epsilon},
\end{equation}
we obtain that
\begin{equation}\label{P5}
\|T_3\|_{X^{1,2}_{k,y_0,\iota\epsilon}}\lesssim |k|^{-3}\big[\log{\langle k\rangle}\big]^7.
\end{equation}
In addition, it follows from Lemma \ref{bX1} and \eqref{L01.2} that
\begin{equation}\label{P6}
\|T_6\|_{Z^1_{k,y_0,\iota\epsilon}}\lesssim |k|^{-2}\big[\log{\langle k\rangle}\big]^7,\qquad\|T_7\|_{Z^1_{k,y_0,\iota\epsilon}}\lesssim |k|^{-4}\big[\log{\langle k\rangle}\big]^7.
\end{equation}
It remains to study $T_2, T_4$ and $T_5$.

{\bf Step 1} In this step we consider the term $T_2$. Using \eqref{P4'} and integration by parts, we obtain
\begin{equation}\label{P7}
\begin{split}
T_2=&\int_0^1\frac{\partial_zG_k(y,z)b''(z)\partial_z\psi^{\iota}_{k,\epsilon}(z,y_0)}{|b'(z)|^2\big(b(z)-b(y_0)+i\iota\epsilon\big)}\,dz+g_{21}\\
      = &   \frac{b''(z)}{(b'(z))^3} \partial_zG_k(y,z) \partial_z\psi^{\iota}_{k,\epsilon}(z,y_0)\log{(b(z)-b(y_0)+i\iota\epsilon\big)}\bigg|_{z=0}^1\\
      &-\int_0^1\frac{\partial_zG_k(y,z)b''(z)}{(b'(z))^3}\partial_z^2\psi^{\iota}_{k,\epsilon}(z,y_0)\log{(b(z)-b(y_0)+i\iota\epsilon\big)}\,dz+g_{21}+g_{22}\\
      =&T_{21}+T_{22}+g_{21}+g_{22}.
\end{split}
\end{equation}
In the above, the more favorable terms $g_{21}, g_{22}$ satisfy
\begin{equation}\label{P8}
\|g_{21}\|_{X^{1,2}_{k,y_0,\iota\epsilon}}+|k|^{-1}\|g_{22}\|_{X^{1,2}_{k,y_0,\iota\epsilon}}\lesssim |k|^{-3}\big[\log{\langle k\rangle}\big]^7.
\end{equation}
$T_{21}$ will be part of the boundary terms in \eqref{P1.0}. Hence it suffices to bound $T_{22}$. Using equation \eqref{F7}, we have
\begin{equation}\label{P9}
\begin{split}
T_{22}=&g_{23}-\int_0^1\frac{\partial_zG_k(y,z)|b''(z)|^2}{(b'(z))^3}\frac{\psi^{\iota}_{k,\epsilon}(z,y_0)}{b(z)-b(y_0)+i\iota\epsilon}\log{(b(z)-b(y_0)+i\iota\epsilon\big)}\,dz\\
            &+\int_0^1\frac{\partial_zG_k(y,z)b''(z)}{(b'(z))^3}\frac{\omega_0^k(z)}{b(z)-b(y_0)+i\iota\epsilon}\log{(b(z)-b(y_0)+i\iota\epsilon\big)}\,dz,
\end{split}
\end{equation}
where
\begin{equation}\label{P10}
\|g_{23}\|_{Y^1_{k,y_0,\iota\epsilon}}\lesssim |k|^{-2}\big[\log{\langle k\rangle}\big]^7.
\end{equation}
Using \eqref{P4'} and integration by parts, we see that
\begin{equation}\label{P11}
T_{22}=\frac{1}{2}\partial_zG_k(y,z)\frac{b''(z)}{(b'(z))^4}\omega_0^k(z)\big(\log{(b(z)-b(y_0)+i\iota\epsilon\big)}\big)^2\bigg|_{z=0}^1+g_{24},
\end{equation}
with 
\begin{equation}\label{P12}
\|g_{24}\|_{X^{1,2}_{k,y_0,\iota\epsilon}}\lesssim |k|^{-2}\big[\log{\langle k\rangle}\big]^7.
\end{equation}

{\bf Step 2} In this step we consider the term $T_4$. Using \eqref{P4'} and integration by parts, and using equation \eqref{F7}, it is easy to obtain that
\begin{equation}\label{P13}
\|T_4\|_{Y^{1,2}_{k,y_0,\epsilon}}\lesssim  |k|^{-3}\big[\log{\langle k\rangle}\big]^7.
\end{equation}

{\bf Step 3} We finally consider the term $T_5$. Using again \eqref{F7}, we can write
\begin{equation}\label{P14}
\begin{split}
T_5:=&\frac{1}{2}\int_0^1\frac{G_k(y,z)b''(z)\partial_z^2\psi^{\iota}_{k,\epsilon}(z,y_0)}{|b'(z)|^2(b(z)-b(y_0)+i\iota\epsilon)}\,dz+g_{51}\\
      =&\frac{1}{2}\int_0^1\frac{G_k(y,z)b''(z)\psi^{\iota}_{k,\epsilon}(z,y_0)}{|b'(z)|^2(b(z)-b(y_0)+i\iota\epsilon)^2}\,dz-\frac{1}{2}\int_0^1\frac{G_k(y,z)b''(z)\omega_0^k(z)}{|b'(z)|^2(b(z)-b(y_0)+i\iota\epsilon)^2}\,dz+g_{51}+g_{52},
\end{split}
\end{equation}
where
\begin{equation}\label{P15}
\|g_{51}\|_{X^{1,2}_{k,y_0,\iota\epsilon}}+|k|^{-1}\|g_{52}\|_{Z^1_{k,y_0,\iota\epsilon}}\lesssim  |k|^{-3}\big[\log{\langle k\rangle}\big]^7.
\end{equation}
Hence by Lemma \eqref{bX17}, 
\begin{equation}\label{P16}
\begin{split}
T_5=&\frac{1}{2}b''(1)\omega_0^k(1)\frac{\sinh{(ky)}}{|b'(1)|^2\sinh{k}}\log{(b(1)-b(y_0)+i\epsilon)}\\
&+\frac{1}{2}b''(0)\omega_0^k(0)\frac{\sinh{(k(1-y))}}{|b'(0)|^2\sinh{k}}\log{(b(0)-b(y_0)+i\epsilon)}+g_5,
\end{split}
\end{equation}
with
\begin{equation}\label{P17}
\|g_{5}\|_{X^{1,2}_{k,y_0,\iota\epsilon}}\lesssim  |k|^{-2}\big[\log{\langle k\rangle}\big]^7.
\end{equation}

Combing the bounds on $T_i, i\in\{1,2,\dots,7\}$ and collecting boundary terms from \eqref{P7}, \eqref{P11} and \eqref{P16}, the lemma is then proved.

\subsection{Proof of Lemma \ref{P18}}
 Using integration by parts, we see that
 \begin{equation}\label{P20}
 \begin{split}
 F^{5\iota}_{k,y_0,\epsilon}=&\int_0^1\partial_zG_k(y,z)\frac{b''(z)}{b'(z)}\frac{\partial_{y_0}\psi^{\iota}_{k,\epsilon}(z,y_0)}{b(z)-b(y_0)+i\iota\epsilon}dz+\int_0^1G_k(y,z)\partial_z\left[\frac{b''(z)}{b'(z)}\right]\frac{\partial_{y_0}\psi^{\iota}_{k,\epsilon}(z,y_0)}{b(z)-b(y_0)+i\iota\epsilon}dz\\
                                             &+\int_0^1G_k(y,z)\frac{b''(z)}{b'(z)}\frac{\partial_z\partial_{y_0}\psi^{\iota}_{k,\epsilon}(z,y_0)}{b(z)-b(y_0)+i\iota\epsilon}dz=T_8+T_9+T_{10}.
 \end{split}
 \end{equation}
 Using \eqref{L0.21}-\eqref{L0.200}, \eqref{P4'} and integration by parts, it is not hard to show that $T_8$ and $T_9$ satisfy the decomposition \eqref{P19.1} with bounds \eqref{P19.0}.
 
  To bound $T_{10}$, we use \eqref{L0.22},  Lemma \ref{bX1} - Lemma \ref{bX17}, and obtain that
 \begin{equation}\label{P23}
 \begin{split}
 T_{10}:=&b'(y_0)\int_0^1G_k(y,z)\frac{b''(z)}{(b'(z))^2}\frac{\omega_0^k(z)}{(b(z)-b(y_0)+i\iota\epsilon)^2}dz\\
         &-b'(y_0)\int_0^1G_k(y,z)\frac{(b''(z))^2}{(b'(z))^2}\frac{\psi^{\iota}_{k,y_0,\epsilon}}{(b(z)-b(y_0)+i\iota\epsilon)^2}dz\\
          &+b'(y_0)\sum_{\sigma\in\{0,1\}}\Phi^{\iota,\ast}_{k,y_0,\epsilon}(y)\omega_0^k(\sigma)\log{(b(\sigma)-b(y_0)+i\iota\epsilon)}+g_{9},
 \end{split}
 \end{equation} 
 where $g_9$ verifies  the decomposition \eqref{P19.1} with bounds \eqref{P19.0}. Applying Lemma \ref{bX17} to \eqref{P23}, the lemma  is then proved.

\end{document}